\documentclass{article}
\usepackage[utf8]{inputenc}

\usepackage{amsmath, amsthm, amssymb}  
\usepackage{mathtools} 
\usepackage{mathdots} 
\usepackage{esint}

\usepackage{subcaption}
\usepackage{tikz}

\usepackage[hypertexnames=false]{hyperref} 

\newcommand*\samethanks[1][\value{footnote}]{\footnotemark[#1]}

\title{Reduced order modeling for elliptic problems with high contrast diffusion coefficients}
\author{Albert Cohen\thanks{Laboratoire Jacques-Louis Lions, Sorbonne Universit\'e, 
4 place Jussieu, 75005 Paris, France (albert.cohen@sorbonne-universite.fr, matthieu.dolbeault@sorbonne-universite.fr and agustin.somacal@sorbonne-universite.fr)}, Wolfgang Dahmen \thanks{Department of Mathematics, University of South Carolina,
1523 Greene St, Columbia, SC 29208, USA (wolfgang.anton.dahmen@googlemail.com)},\\
Matthieu Dolbeault\samethanks[1], Agustin Somacal\samethanks[1]}
\date{\today}

\newcommand\eps{\varepsilon}
\newcommand\e{\varepsilon}
\renewcommand\t{\tilde}
\newcommand\wt{\widetilde}
\newcommand\<{\langle}
\renewcommand\>{\rangle}
\renewcommand\leq{\leqslant}
\renewcommand\geq{\geqslant}

\def\cM{{\cal M}}
\def\cN{{\cal N}}

\def\cB{{\cal B}}
\def\cL{{\cal L}}
\def\cK{{\cal K}}

\def\o{\overline}
\def\({\left (}
\def\){\right )}

\newcommand{\R}{\mathbb R}
\newcommand{\N}{\mathbb N}
\def\Chi{\raise .3ex\hbox{\large $\chi$}}

\newcommand{\be}{\begin{equation}}
\newcommand{\ee}{\end{equation}}

\newcommand{\iref}[1]{{\rm (\ref{#1})}}

\newtheorem{theorem}{Theorem}[section]
\newtheorem{proposition}[theorem]{Proposition}
\newtheorem{lemma}[theorem]{Lemma}
\newtheorem{remark}[theorem]{Remark}
\newtheorem{definition}[theorem]{Definition}

\DeclareMathOperator{\divergence}{div}

\DeclareMathOperator{\Span}{span}

\newcommand{\ac}[1]{\textcolor{black}{#1}}

\begin{document}

\maketitle

\begin{abstract}
\ac{
We consider a parametric elliptic PDE with a scalar piecewise constant diffusion coefficient taking arbitrary positive values 
on fixed subdomains. This problem is not uniformly elliptic, as
the contrast can be arbitrarily high, contrarily to the Uniform Ellipticity Assumption (UEA) that is commonly made
on parametric elliptic PDEs. We construct reduced model spaces
that approximate uniformly well all solutions with estimates in relative error
that are independent of the contrast level. These estimates are sub-exponential
in the reduced model dimension, yet exhibiting the curse of dimensionality as the number
of subdomains grows. Similar estimates are obtained for the Galerkin projection,
as well as for the state estimation and parameter estimation inverse problems. 
A key ingredient in our construction and analysis is the study
of the convergence towards limit solutions of stiff problems when diffusion tends to 
infinity in certain domains.}
\end{abstract}

\section{Introduction}

\subsection{Reduced models for parametrized PDEs}

Parametric PDE's are commonly used to describe complex physical phenomena. 
\ac{With $y=(y_1, \dots, y_d)$ denoting a parameter vector ranging in some domain $Y\subset \R^d$,
and $u(y)$ the corresponding solution to the PDE of interest, assumed
to be well defined in some Hilbert space $V$, we denote by
\be
\cM:=\{u(y)\;: \; y\in Y\},
\label{manif}
\ee
the collection of all solutions, called the {\it solution manifold}.}

There are two main ranges of problems
associated to parametric PDEs:
\begin{enumerate}
\item
Forward modeling: in applications where many queries
of the parameter to solution map $y \mapsto u(y)$ are required, one needs
numerical forward solvers that efficiently compute approximations $\t u(y)$ with
a prescribed accuracy.
\item
Inverse problems: when the exact value of the parameter $y$ is unknown,
one is interested in either recovering an approximation
to $u(y)$ (state estimation) or to $y$ (parameter estimation), 
from a limited number of observations $z_i=\ell_i(u(y))$,
possibly corrupted by noise.
\end{enumerate}

{\it Reduced order modeling} is widely used for tackling both problems.
In its most common form, its aim is to construct linear spaces $V_n$
of moderate dimension $n$ that approximate all solutions $u(y)$ with
best possible certified accuracy.  The natural benchmark for measuring the performance of such linear reduced models is provided by
the {\it Kolmogorov $n$-width} of the solution manifold
\be
d_n(\cM)_V:=\inf_{\dim(V_n)=n} {\rm dist}(\cM,V_n)_V
\label{nwidth}
\ee
that describes the performance of an optimal space. Here
\[
{\rm dist}(\cM,V_n)_V:=\sup_{u\in \cM}\inf_{v\in V_n} \|u-v\|_V=\sup_{u\in \cM}\|u-P_{V_n}u\|_V,
\]
where $P_{V_n}$ is the $V$-orthogonal projector onto $V_n$. We refer the reader to 
\cite{P} for a general treatment of $n$-widths.

While an optimal space achieving the above infimum
is usually out of reach, there exist two main approaches
aiming to construct ``sub-optimal yet good'' spaces. The first one consists
in building expansions of the parameter to solution map, \ac{for example by polynomials
\be
u_n(y):=\sum_{\nu\in \Lambda_n} u_\nu y^\nu, \quad y^\nu:=y_1^{\nu_1}\dots y_d^{\nu_d},
\label{pol}
\ee
where $\Lambda_n\subset \N^d$ is a set of cardinality $n$. The coefficients $u_\nu$ are
elements of $V$ and therefore, for all $y\in Y$ the approximation $u_n(y)$ is picked from the space
\[
V_n:={\rm span}\{u_\nu\; : \; \nu\in \Lambda_n\}.
\]
Notice that $u_n(y)$ is not the orthogonal projection $P_{V_n}u(y)$ in this case, but $u_n(y)$ is  easy to compute for a given query $y$ once 
the $u_\nu$ have been constructed 
(usually through a high fidelity finite element solver). We refer to \cite{BNT,BNTT1,BNTT2,BCM,CD,CDS,TWZ} for instances of this approach.}

The second approach is the reduced basis method \ac{\cite{Ha,RHP,S}, that consists in taking}
\[
V_n:={\rm span}\{u^1, \dots, u^n\},
\]
where the $u^j=u(y^j)$ are particular solution instances
corresponding to a selection of parameter vectors $y^j\in Y$. \ac{A close variant is the proper orthogonal decomposition
method \cite{Ch,Vo,WP}, where the reduced spaces are obtained by principal component analysis applied to large training 
set of such instances. In the reduced basis method,} the parameter vectors $y^1, \dots, y^n$ can be selected by a greedy algorithm, \ac{introduced in \cite{VPRP}} and
originally studied in \cite{BMPPT}. For such a selection process, it is proved in \cite{BCDDPW, DPW} that if
$d_n(\cM)_V$ has a certain algebraic or exponential rate of decay with $n$,
then a similar rate is achieved by ${\rm dist}(\cM,V_n)_V$ for the reduced basis
spaces. 

It follows that the  reduced basis spaces constructed by the greedy algorithm  are close to optimal.  This is in contrast to
the spaces $V_n$ spanned by the polynomial coefficients $u_\nu$ for which
the approximation rate is not guaranteed to be optimal. \ac{We refer to \cite{BC} for instances where reduced basis methods
can be proved to converge with a strictly higher rate than polynomial approximations.}
On the other hand, the polynomial constructions \eqref{pol} have certain numerical advantages.   Namely, \ac{for several relevant classes of
parametrized PDEs,}
it can be shown that  the parameter to solution mapping $y\mapsto u(y)$  has certain \ac{smoothness properties}  that can be used to obtain a-priori bounds on the $\|u_\nu\|_V$ without actually computing these norms.
This allows  an  a priori
selection of  an appropriate set $\Lambda_n$ and the proof of concrete approximation 
estimates for the error $\sup_{y\in Y} \|u(y)-u_n(y)\|_V$. These estimates in turn provide an upper bound for $d_n(\cM)_V$, and therefore for
reduced basis approximations.

\subsection{Parametrized elliptic PDEs}

One prototypal instance where the convergence analysis described above has been 
deeply studied is the \ac{parametrized second order elliptic equation
\be
-\divergence (a(y)\nabla u(y)) = f \quad {\rm in} \; \Omega, \quad u_{|\partial \Omega}=0
\quad {\rm on} \;\partial \Omega,
\label{ellip}
\ee
where $\Omega\subset \R^m$ is the spatial domain, $f\in H^{-1}(\Omega)$ is a source term, and $a(y)$ has the {\it affine} form}
\be
a(y)=\o a +\sum_{j=1}^dy_j\psi_j,
\label{affine}
\ee
with $\o a$ and $(\psi_1,\dots,\psi_d)$ some fixed functions in $L^\infty(\Omega)$. 

\ac{The corresponding solution $u(y)\in H^1_0(\Omega)$ is defined through the standard variational formulation
in $H^1_0(\Omega)$ equipped with its usual norm. }Up to renormalization, it is usually assumed that the $y_j$ range in $[-1,1]$, 
or equivalently $Y=[-1,1]^d$. To ensure existence and uniqueness of solutions, one typically assumes that the so-called {\it Uniform Ellipticity Assumption}
(UEA) holds: for some fixed $0<r\leq R<\infty$,
\be
r\leq a(x, y) \leq R, \quad \quad x\in \Omega, \quad y\in Y,
\label{uea}
\ee
where $a(x,y):=a(y)(x)=\o a(x) +\sum_{j=1}^dy_j\psi_j(x)$,
or in short $r\leq a(y) \leq R$ for all $y\in Y$. Under this assumption, Lax-Milgram theory ensures that the solution map
$y\mapsto u(y)$ is well defined from $Y$ into $H^1_0(\Omega)$,
with the uniform bound
\[
\|u(y)\|_{H^1_0}:=\|\nabla u(y)\|_{L^2} \leq \frac {C_f} {r}, \quad y\in Y.
\]
Here and throughout this paper
\be
C_f:=\|f\|_{H^{-1}}.
\label{Cf}
\ee
It was proved in \cite{BNTT1,TWZ}
that, under UEA, polynomial approximations \iref{pol} of given total degree 
converge \ac{sub-exponentially: for $\Lambda_n=\{|\nu|\leq k\}$ with $n={k+d\choose d}$, one has
\be
\sup_{y\in Y} \|u(y)-u_n(y)\|_{H^1_0} \leq C'\exp(-cn^{1/d}),
\label{expn}
\ee
Such sub-exponential} rates show that 
the spaces $V_n$ based on polynomial expansions or reduced bases
perform significantly better than standard finite element spaces, at least
for a moderate number $d$ of parameters. \ac{It is possible
to maintain a rate of convergence as $d$ grows, and even when $d=\infty$,
when assuming some anisotropy in the variable $y_j$ through the decay
of the size of $\psi_j$ as $j\to \infty$, see in particular \cite{BCM,CD,CDS} for results of this type.}

\subsection{High constrast problems}

The Uniform Ellipticity Assumption \iref{uea} implies that there
is a uniform control on the level of contrast in the diffusion function
\be
\kappa(y):= 
\frac{\max_{x\in \Omega} a(x,y)}{\min_{x\in \Omega} a(x,y)} \leq \frac R r, \quad y\in Y.
\label{contrast}
\ee
This assumption also plays a key role in the derivation of the above approximation results, since it 
guarantees that the parameter
to solution map has a holomorphic extension to a sufficiently
large complex neighbourhood of $Y$. In this case, a good polynomial approximation
$u_n$ may be defined by simply truncating \ac{the power series $\sum_{\nu\in\N^d} u_\nu y^\nu$, leading to the estimate \iref{expn}.
}

On the other hand, there exist various situations where 
one would like to avoid such a strong restriction on the level of contrast.
Perhaps the most representative setting is when the domain $\Omega$
is partitioned into disjoint subdomains 
$\{\Omega_1,\dots,\Omega_d\}$, each of them admitting
a constant diffusivity level that could vary strongly between subdomains.
This is typically the case when modeling diffusion in materials having
multiple layers or inclusions that could have very different nature, \ac{for example
air or liquid versus solid. This situation can be encountered in groundwater flow applications, 
where certain subdomains correspond to cavities, for which the diffusion function
becomes nearly infinite, as opposed to subdomains containing sediments or other porous rocks.}

In such a case, we do not want to limit the contrast level. To represent this setting, we let 
\be
a(y)_{|\Omega_j}=y_j, \quad y_j\in ]0,\infty[
\label{pw}
\ee
or equivalently $a(y)=\sum_{j=1}^d y_j\Chi_{\Omega_j}$,
which corresponds to the affine form \iref{affine} with $\o a=0$ and $\psi_j=\Chi_{\Omega_j}$, now with
\be
Y:=]0,\infty[^d.
\label{parameterdomain}
\ee
We take \eqref{parameterdomain} as the definition of the parameter domain $Y$ for the remainder of this paper. 
The solution $u(y)$ satisfies the variational
formulation
\be
\sum_{j=1}^d \,y_j\int_{\Omega_j} \nabla u(y)\cdot\nabla v \, dx=\<f,v\>_{H^{-1},H^1_0}, \quad v\in H^1_0(\Omega),
\label{variat}
\ee
or equivalently $-y_j\Delta u(y)=f$ as elements of $H^{-1}(\Omega_j)$ on each $\Omega_j$, with the standard jump conditions $[a(y)\partial_{\vec{n}} u(y)]=0$ across the boundaries between subdomains.

Let us observe that in this setting, it is hopeless to find spaces $V_n$ that
approximate all solutions $u(y)$ uniformly well. Indeed, the following 
homogeneity property obviously holds: for any $y\in Y$ and $t>0$, one has
\be
u(t y)=t^{-1}u(y).
\label{homog}
\ee
This property implies in particular that $\|u(y)\|_{H^1_0}$ tends to infinity 
as $y\to 0$, and so does $\|u(y)-P_{V_n}u(y)\|_{H^1_0}$ in general. In fact, this also shows
that the solution manifold $\cM$ is {\it not} relatively compact and does
not have finite $n$-widths.

In addition to this principal difficulty, let us remind that when using the spaces
$V_n$ in forward modeling, we typically use the Galerkin method, \ac{that delivers
the orthogonal projection onto $V_n$ however for the energy norm
\be
\|v\|_{y}^2:=\sum_{j=1}^d \,y_j \int_{\Omega_j} |\nabla v|^2 \, dx.
\label{ynorm}
\ee
This approximation is thus optimal in $H^1_0(\Omega)$,
however up to the constant $\kappa(y)^{1/2}$, which deteriorates with high contrast.}
\newline

{\it The main contribution of this paper is to treat these issues, and derive approximation estimates that are robust to high contrast, in the sense that
they are independent of $y\in Y$.} 
\newline

Due to the
main objection coming from the homogeneity property \iref{homog}, it is 
natural to look for uniform approximation estimates in relative error, that is, estimates of the form
\be
\label{rel}
\|u(y)-P_{V_n}u(y)\|_{H^1_0}\leq \eps_n \|u(y)\|_{H^1_0}, \quad y\in Y,
\ee
with $\lim_{n\to\infty}\eps_n=0$, and similarly for $P_{V_n}^yu(y)$. 
Our main results, Theorems \ref{width} and \ref{widthy}, exhibit spaces $V_n$ ensuring the validity of such uniform 
estimates with $\eps_n$ having sub-exponential decay with $n$,
similar to the known results under UEA.

\begin{remark}
\ac{High contrast problems have been the object of intense investigation, in particular 
with the objective of developing techniques for multilevel or domain decomposition preconditioning \cite{AY,AGKS,GE} 
and a-posteriori error estimation \cite{Ain,BV}, that are provably robust with respect to the level of contrast.
To our knowledge, the present work is the first in which this robustness is established for
reduced modeling methods.}
\end{remark}

\subsection{Outline}

Throughout this paper, we consider the parametrized elliptic PDE \iref{ellip} with 
$a(y)$ having piecewise constant form \iref{pw} over a fixed partition.
In view of  the homogeneity property \iref{homog}, we are led to consider
the subset  
\be
\label{pd}
Y':= [1,\infty[^d
\ee
of parameters corresponding to the coercive regime. Any result on relative approximation error that is established for $Y'$ extends automatically to all of $Y$ because of the homogeneity property.
Accordingly, we let 
\be
\cB:=\{u(y)\; : \; y\in Y' \}.
\ee

In \S 2, we start by proving that $\cB$ is a precompact set of $H^1_0(\Omega)$.
\ac{One crucial ingredient for this analysis \ac{are the} {\it limit
solutions} of the so-called {\it stiff problem}, obtained as $y_j\to \infty$ for certain $j\in \{1,\dots,d\}$. }

In \S 3, we construct specific reduced model spaces for which
the approximation estimate \iref{rel} holds with $\eps_n$ decaying \ac{sub-exponentially}.
\ac{Our construction is based on partitioning the parametric 
domain $Y'$ into rectangular regions and using a
different polynomial approximations on each region. This results
in global reduced model space $V_n$ for which the accuracy bound
remains \ac{sub-exponential}, however in $\exp(-cn^{\frac 1{2d-2}})$. 
A key ingredient for establishing these \ac{sub-exponential} rates is the derivation of 
quantitative estimates on the convergence of $u(y)$ towards 
limit solutions defined in \S 2 as some $y_j$ tend to infinity. These estimates
are established under an additional geometrical assumption on the partition,
similar results for a general partition of $\Omega$ being an open problem.}

In \S 4, we discuss the use of these reduced model spaces in forward modeling and inverse problems. Our main result relative to forward modeling
is that the estimate \iref{rel}
also holds for the Galerkin projection with 
the same exponential decay $\e_n$. We show that such a result
is only possible if $V_n$ includes functions that have constant values over some subdomains. For the state estimation problem, 
we follow the Parametrized Background Data Weak (PBDW)
method \cite{BCDDPW17,MPPY}, 
and obtain recovery bounds that are uniform over $y\in Y$ in relative error.
For the parameter estimation problem, we introduce an ad-hoc strategy that
specifically exploits the piecewise constant structure of the diffusion coefficient and obtain similar recovery bounds for the inverse diffusivity. 

We conclude in \S 5 by presenting some numerical illustrations revealing the effectiveness of the reduced model spaces
even in the high-contrast regime, as expressed by the approximation results.
\newline
\newline
\ac{{\bf Acknowledgements:} We thank the anonymous reviewers for their constructive comments.
We also thank Fran\c cois Murat for useful discussions in the understanding of the convergence process towards limit solutions,
Hamza Maimoune for leading us to this work through
his remarks during his master project, and Jules Pertinand for useful discussions.
}

\section{Uniform approximation in relative error}

In this section we work under no particular geometric  
assumption on the partition $\{\Omega_1,\dots,\Omega_d\}$
of $\Omega$, and consider the solution manifold \ac{$\cM$ defined by \iref{manif}},
where $u(y)\in H^1_0(\Omega)$ is solution to the elliptic boundary value problem with variational formulation \iref{variat}. Our objective is to
show the existence of spaces $V_n$ that uniformly approximate
$\cM$ in the relative error sense expressed by \iref{rel}. 

\subsection{\ac{Limit solutions} and the extended solution manifold}

Our first observation is that this collection can be 
continuously extended when $y_j=\infty$ for some values
of $j$, \ac {through limit solutions of stiff inclusions problems.
Such limit solutions have for example been
considered in the context homogeneization, see e.g. p.98 of \cite{JKO}.}

For this purpose, to any $S\subset \{1,\dots,d\}$, we associate the
space 
\be
V_S:=\{v\in H^1_0(\Omega) \; : \; \nabla v_{| \Omega_j}=0, \;\; j\in S\}.
\ee
In other words, $V_S$ consists of the functions from $H^1_0(\Omega)$
that have constant values on the subdomains $\Omega_j$ for $j\in S$ (or on each of their connected components if these subdomains are not connected).
It is a closed subspace of $H^1_0(\Omega)$.
We decompose the parameter vector $y$ according to
\be
y=(y_S,y_{S^c}), \quad y_S:=(y_j)_{j\in S}\quad\text{and} \quad y_{S^c}:=(y_j)_{j\in S^c}.
\ee
For any finite and positive vector $y_{S^c}$,
similar to the $\|\cdot\|_y$ norm \iref{ynorm}, we may define
\be
\|v\|_{y_{S^c}}^2:=\sum_{j\in S^c} \,y_j \int_{\Omega_j} |\nabla v|^2 \, dx,
\ee
which is a semi-norm on $H^1_0(\Omega)$, and a full norm equivalent to 
the $H^1_0$-norm on $V_S$. Also note that when $y=(y_S,y_{S^c})$ is finite,
one then has $\|v\|_{y_{S^c}}=\|v\|_{y}$ for any $v\in V_S$.

For any finite and positive vector $y_{S^c}$, we 
define the function $u_S(y_{S^c}) \in V_S$ solution to the 
following \ac {stiff inclusions} problem:
\be
\sum_{j\in S^c} \,y_j \int_{\Omega_j} \nabla u_S(y_{S^c})\cdot \nabla v \, dx=\< f, v \>_{H^{-1},H^1_0}, \quad v\in V_S.
\label{limitproblem}
\ee
The following result shows that this solution is well defined and 
is the limit of $u(y)$, when $y_{S^c}$ is fixed and $y_j\to \infty$ for $j\in S$. \ac{Note that the
weak convergence is established in \cite{JKO} (p. 98) and so we concentrate the proof on the strong convergence.}

\begin{lemma}
\label{limlemma}
There exists a unique $u_S(y_{S^c})\in V_S$ solution to \iref{limitproblem},
which is the limit in $H^1_0(\Omega)$ of the solution 
$u(y_{S}, y_{S^c})$ as $y_j\to \infty$ for all $j\in S$.
\end{lemma}

\begin{proof} \ac{Using the bilinear form $(u,v)\mapsto \sum_{j\in S^c}\,  y_j \int_{\Omega_j}\nabla u\cdot\nabla v\, dx$
in the space $V_S$, Lax-Milgram theory implies the existence of a unique solution $u_S(y_{S^c})\in V_S$ to \iref{limitproblem}.}

\ac{ Consider now a sequence $(y^n)_{n\geq 1}\in Y^{\mathbb N}$, with $y^n_{S^c}=y_{S^c}$ and $y^n_j\to\infty$ for all $j\in S$. Denoting $u_n=u(y^n)$,
 it is readily seen that $(u_n)_{n\geq 1}$ is uniformly bounded in $H^1_0$ norm by $C=C_f\,c^{-1}$, where $c:=\min_{n\geq 1}\min_{1\leq j \leq d} y^n_j>0$,
 and that any weak limit of a sequence extraction is solution to the variational equation \iref{limitproblem}. Therefore the whole sequence $(u_n)_{n\geq 1}$ weakly converges to $\bar u=u_S(y_{S^c})$. 
 }
 
We finally prove strong convergence by writing
\begin{align*}
c \|u_n-\bar u\|_{H^1_0}^2
&\;\;\leq \,\int_\Omega a(y^n)|\nabla(u_n-\bar u)|^2\, dx\\[3.5pt]
&\;\;= \,\<f,u_n\>_{H^{-1},H^1_0}-2\<\bar u, u_n\>_{y_{S^c}}+\|\bar u\|_{y_{S^c}}^2\\[7pt]
&  \underset{n\to\infty}{\longrightarrow}  \<f,\bar u\>_{H^{-1},H^1_0} -\|\bar u\|_{y_{S^c}}^2=0.
\end{align*}
\end{proof}

The above lemma allows us to readily extend 
the solution manifold by introducing
\[
\wt Y:= ]0,\infty]^d,
\]
and
\[
\o \cM:=\{u(y)\; : \; y\in\wt Y\},
\]
where we have formally set
\[
u(y):=u_S(y_{S^c}),
\]
when $y_j=\infty$ for $j\in S$ and $y_j<\infty$ for $j\in S^c$.
Note that when $S=\{1,\dots,d\}$ the space $V_S$ is trivial and 
one has
\[
u(\infty,\dots,\infty)=0.
\]
\begin{remark}
Although we do not make explicit use of it, it can be
checked that despite the fact that $y_j=0$ is excluded in the
definition of $\o \cM$, it indeed coincides with the closure of $\cM$
in $H^1_0(\Omega)$ due to the fact that $\|u(y)\|_{H^1_0}\to \infty$
as $y\to 0$.
\end{remark}

\begin{remark}
\ac{More precisely, when some $y_j$ tend to zero, $u(y)$ converges to the solution of the so-called soft inclusions problem (see \cite{JKO}, chapter 3), outside the corresponding subdomains $\Omega_j$. Here, due to the fact that the approximation estimates that we prove further are in relative error, 
these other limit solutions are of no use in our analysis.}
\end{remark}

\subsection{A compactness result}

As already observed in the introduction, the manifold $\o\cM$ is not bounded in $H^1_0(\Omega)$ due to the homogeneity property \iref{homog} and therefore not compact.

In order to treat this defect, we consider
\[
\wt Y':=[1,\infty]^d,
\]
and the submanifold
\[
\o\cB:=\{u(y)\; : \; y\in \wt Y'\},
\]
which is now bounded in $H^1_0(\Omega)$, 
from the standard a-priori estimate
\[
\|u(y)\|_{H^1_0} \leq \frac{C_f}{\min y_j} \leq C_f,
\]
that is obtained by taking $v=u(y)$ in the variational formulation 
\iref{variat}, with $C_f=\|f\|_{H^{-1}}$ as in \iref{Cf}. This estimate trivially extends to $u_S(y_{S^c})$
when the $y_j$ have infinite value for $j\in S$. 
In addition we have the following result.

\begin{theorem}
\label{comptheo}
The set $\o\cB$ is compact in $H^1_0(\Omega)$.
\end{theorem}

\begin{proof}
Consider any sequence of vectors $y^n=(y_1^n,\dots,y_d^n)\in \wt Y'$
for $n\geq 1$. We need to prove that the corresponding sequence of solutions
$
(u(y^n))_{n\geq 1}
$
admits a converging subsequence. For this purpose, we observe
that there exists a subset $S\in \{1,\dots,d\}$ such that, up to subsequence
extraction,
\[
\lim_{n\to \infty} y_j^n =\infty, \quad j\in S,
\]
and
\[
\lim_{n\to \infty} y_j^n = y_j<\infty, \quad j\in S^c.
\]
Note that $S$ could be empty, for instance in the case where the $y_j^n$
are uniformly bounded for all $j$.

Let $\e>0$. \ac{Using the strong convergence result in Lemma \ref{limlemma}, for all $n\geq 1$ there exists
an auxiliary vector $\bar y^n$ such that
$\bar y_j^n=y_j^n$ when $y_j^n<\infty$,
$\bar y_j^n<\infty$ when $y_j^n=\infty$, such that by having picked $\bar y_j^n$ large enough in the second case}
\[
\|u(y^n)-u(\bar y^n)\|_{H^1_0}\leq \e/3.
\]
In addition we may assume that $\bar y_j^n\to \infty$ for
$j\in S$. Next we introduce the vector 
$\t y^n$ such that $\t y_j^n=\bar y_j^n$ when $j\in S$ and
$\t y_j^n=y_j$ when $j\in S^c$. Applying again Lemma \ref{limlemma}, we find that with $y_{S^c}=(y_j)_{j\in S^c}$,
one has
\[
\|u(\t y^n)-u_S(y_{S^c})\|_{H^1_0} \leq \e/3,
\]
for $n$ sufficiently large. Finally we argue that 
\[
\|u(\t y^n)-u(\bar y^n)\|_{H^1_0} \leq \e/3,
\]
for $n$ large enough. This is a consequence of \ac{the following variant of
Strang first lemma (which proof is similar and left as an exercise to the reader)} that says 
that for two diffusion functions $\bar a$ and $\t a$, the corresponding solution $\bar u$ and $\t u$ with the
same data $f$ satisfy
\[
\|\bar u-\t u\|_{H^1_0}\leq\frac{C_f\,\|\bar a-\t a\|_{L^\infty}}{ \min\{\bar a_{\min}, \t a_{\min}\}^{2}}.
\]
\ac{We then apply this to $\o a:=\o a_n=a(\o y^n)$ and $\t a:=\t a_n=a(\t y^n)$, 
observing that from their definition, $\|\bar a-\t a\|_{L^\infty}=\max_{j\in S^c} |\bar y^n_j-y_j|  \to 0$ as $n\to \infty$.}
Therefore $\|u(y^n)-u_S(y_{S^c})\|_{H^1_0}\leq \e$ for $n$ sufficiently large, which concludes the proof.
\end{proof}

We next observe that any $y\in Y$ can be rewritten as 
\[
y=t \t y,
\]
with $\t y\in Y'$ and normalization $\min \t y_j=1$, for some $t>0$,
and from \iref{homog} one has $u(y)=t^{-1}u(\t y)$. This motivates 
the study of the further reduced manifold
\be
\cN:=\{u(y)\; : \; y\in \wt Y',\; \min y_j=1\},
\label{normmanifold}
\ee
which is a subset of $\o \cB$. 

One important observation is that the solutions contained in $\cN$
are also uniformly bounded from below, under mild assumptions
on the data $f$.

\begin{lemma}
The set $\cN$ is compact in $H^1_0(\Omega)$. Moreover, one has the framing
\be
\min_{1\leq j \leq d} \|f\|_{H^{-1}(\Omega_j)} \leq \|u(y)\|_{H^1_0}\leq C_f, 
\label{framing}
\ee
for all $u(y)\in \cN$.
\end{lemma}

\begin{proof}
The compactness of $\cN$ follows from that of $\o\cB$, since $\cN$ is a closed subset of $\o\cB$. For the framing, as $a(y)\geq 1$ on $\Omega$,
\[
\|u\|_{H^1_0}^2
\leq \sum_{j\in S^c}\,y_j\int_{\Omega_j} |\nabla u(y)|^2\, dx=\<f,u(y)\>_{H^{-1},H^1_0}\leq C_f\|u(y)\|_{H^1_0},
\]
so $\|u(y)\|_{H^1_0}\leq C_f$. Now take $j\in\{1,\dots,d\}$ such that $y_j=1$, and consider $\phi\in H^1_0(\Omega_j)$. Then
\[
\<f,\phi\>_{H^{-1},H^1_0}=\int_{\Omega_j} \nabla u(y)\cdot\nabla \phi\, dx\leq \|u(y)\|_{H^1_0(\Omega)}\|\phi\|_{H^1_0(\Omega_j)},
\]
\ac{which gives the result.}
\end{proof}

In the sequel of this paper, we always work under the condition that
the lower bound in \iref{framing} is strictly positive
\be
c_f:=\min_{1\leq j \leq d} \|f\|_{H^{-1}(\Omega_j)}>0.
\label{cf}
\ee
Let us observe that when $f$ is a function in $L^2(\Omega)$,
this is ensured as soon as $f$ is not identically zero
on one of the $\Omega_j$. We thus have 
\be
0<c_f\leq \|u(y)\|_{H^1_0} \leq C_f,
\label{frameh10}
\ee
for all $u(y)\in \cN$.

\begin{remark}
The condition $c_f>0$ is in general 
necessary for controlling $\|u(y)\|_{H^1_0}$ from below. Indeed assume $\|f\|_{H^{-1}(\Omega_j)}=0$
for some $j$ such that $\o \Omega\setminus\o\Omega_j$ is connected. Then taking $y_k=\infty$ for $k\neq j$ and $y_j=1$, we find that $u(y)\in V_S$ 
with $S=\{j\}^c$, which is equivalent to $u(y)\in H^1_0(\Omega_j)$ 
since it vanishes on the other sub-domains. As $\|f\|_{H^{-1}(\Omega_j)}=0$, we obtain $u(y)=0$.
\end{remark}

\begin{remark}
\label{remframey}
One also has the uniform framing in the $\|\cdot\|_y$ norm since
\be
0<c_f\leq \|u(y)\|_{H^1_0}\leq \|u(y)\|_{y} =\sqrt{\<f,u\>_{H^{-1},H^1_0}} \leq C_f,
\label{framey}
\ee
for all $u(y)\in \cN$ when all $y_j$ are finite. 
\end{remark}

The framing \iref{frameh10} has an implication on the existence of reduced model
spaces that approximate uniformly well all solutions $u(y)\in \o \cM$ in relative error.

\begin{theorem}
\label{theorelprojerror}
There exists a sequence of linear spaces $(V_n)_{n\geq 1}$ such
that $\dim(V_n)=n$, and a sequence $(\eps_n)_{n\geq 1}$ that
converges to zero such that
\be
\|u(y)-P_{V_n}u(y)\|_{H^1_0} \leq \eps_n \|u(y)\|_{H^1_0}
\label{releps}
\ee
for all $y\in \wt Y$, where $P_{V_n}$ is the $H^1_0(\Omega)$-orthogonal 
projector onto $V_n$.
\end{theorem}

\begin{proof} Since $\cN$ is compact, there exists
a sequence of spaces $(V_n)_{n\geq 1}$ with
$\dim(V_n)=n$ and a sequence $(\sigma_n)_{n\geq 1}$
that tends to $0$, such that
\[
\|v-P_{V_n}v\|_{H^1_0}\leq \sigma_{n}, \quad v\in \cN.
\]
Now let $y\in\wt Y$ differing from
$(\infty,\dots,\infty)$, for which there is nothing to prove
since $u(\infty,\dots,\infty)=0$, and let $t^{-1}=\min_{1\leq j \leq d} y_j<\infty$. By homogeneity, $t^{-1}u(y)=u(ty)\in \cN$, 
and therefore
\[
\|u(y)-P_{V_n}u(y)\|_{H^1_0}=t\|u(ty)-P_{V_n}u(ty)\|_{H^1_0(\Omega)}\leq t \sigma_n.
\]
On the other hand, $\|u(y)\|_{H^1_0(\Omega)}=t\|u(ty)\|_{H^1_0(\Omega)}\geq tc_f$ by framing \iref{framing}, which proves Theorem \ref{theorelprojerror} with $\eps_n=\sigma_{n}/c_f$.
\end{proof}

The above theorem tells us that we can achieve 
contrast-independent approximation in relative
error. It is however still unsatisfactory from two perspectives:
\begin{enumerate}
\item
It does not describe the rate of decay of $\eps_n$ as 
the reduced dimension $n$ grows. In practice, one would
like to construct reduced spaces $V_n$ such that this decay
is fast, similar to the exponential decay obtained under UEA.
\item
The approximation property is expressed in terms of the
orthogonal projection $P_{V_n}$. In applications to 
forward modeling, we approximate the
solution $u(y)$ in the space $V_n$ by the Galerkin 
projection $P_{V_n}^y u(y)$. We thus wish for uniform estimates also for such approximations.
\end{enumerate}

These two problems are treated in \S 3
and \S 4 respectively.

\section{Approximation rates}

Our construction of efficient reduced model spaces is based
on a certain partitioning of the parameter domain $\wt Y'$
associated to the manifold $\o\cB$. To any $\ell=(\ell_1,\dots,\ell_d)\in\N_0^d$
we associate the dyadic rectangle 
\be
R_{\ell}=[2^{\ell_1},2^{\ell_1+1}]\times \dots \times [2^{\ell_d},2^{\ell_d+1}],
\ee
For a positive integer $L$ to be fixed further, we modify the definition
of $R_{\ell}$ by replacing the interval $[2^{\ell_j},2^{\ell_j+1}]$ by $[2^{\ell_j},\infty]$
when $\ell_j=L$ for some $j$. This leads to the partition
\be
\wt Y'=\bigcup_{\ell\in \{0,\dots,L\}^d} R_{\ell}.
\ee

\begin{figure}
\begin{center}
\begin{tikzpicture}
  \draw[->](0,0)--(5,0);
  \draw[->](0,0)--(0,5);
  \draw(0,4)--(4,4);
  \draw(4,0)--(4,4);
  \draw(0,2)--(4,2);
  \draw(2,0)--(2,4);
  \draw(0,1)--(4,1);
  \draw(1,0)--(1,4);
  \draw(0,0.5)--(4,0.5);
  \draw(0.5,0)--(0.5,4);
  \node at (4.7,-0.3) {$z_1$};
  \node at (-0.3,4.7) {$z_2$};
  \node at (-0.3,-0.3) {$0$};
  \node at (-0.3,4) {$1$};
  \node at (4,-0.3) {$1$};
  \node at (-0.3,2) {$\frac{1}{2}$};
  \node at (2,-0.3) {$\frac{1}{2}$};
  \node at (-0.3,0.5) {$\frac{1}{2^L}$};
  \node at (0.5,-0.3) {$\frac{1}{2^L}$};
  \node at (3,3) {$R_{11}^{-1}$};
  \node at (3,1.5) {$R_{12}^{-1}$};
  \node at (1.5,3) {$R_{21}^{-1}$};
  \node at (1.5,1.5) {$R_{22}^{-1}$};
  \node at (3,0.85) {$\vdots$};
  \node at (3,0.25) {$R_{1L}^{-1}$};
  \node at (0.75,3) {$...$};
  \node at (0.75,0.85) {$\iddots$};
  \node at (-0.3,1.25) {$\vdots$};
  \node at (1.25,-0.3) {$\dots$};
\end{tikzpicture}
\end{center}
\caption{Partition of $[0,1]^d$ by the inverse rectangles $R_{\ell}^{-1}$ in the case $d=2$.}
\label{figure1}
\end{figure}
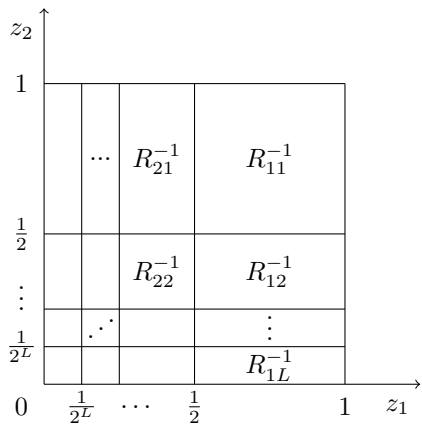

This partition is best visualized in the inverse parameter domain by setting
\be
z=(z_1,\dots,z_d):=(y_1^{-1},\dots,y_d^{-1})\in [0,1]^d.
\ee
Then, the inverse rectangles $R_{\ell}^{-1}$ split the unit cube, as shown on
Figure \ref{figure1}. In particular, the rectangles touching the axes correspond
to rectangles $R_{\ell}$ of infinite size. 

We build reduced model spaces through a piecewise
polynomial approximation over this partition. In other words, for each 
$\ell\in\{0,\dots,L\}^d$, we use different polynomials 
\[
u_{\ell,k}(y)=\sum_{|\nu|\leq k}u_{\ell,\nu}y^\nu,
\]
of total degree $k$ for approximating $u(y)$ when $y\in R_{\ell}$, leading to a family of local reduced model spaces
\be
\label{reducedlocal}
V_{\ell,k}={\rm span}\{u_{\ell,\nu}\; : \; |\nu|\leq k\},
\ee
that can be either used individually when approximating $u(y)$ if
the rectangle $R_\ell$ containing $y$ is known, or summed up
in order to obtain a global reduced model space.

In this section we show that this construction yields exponential 
convergence rates in \iref{rel}, similar to those obtained under 
a Uniform Ellipticity Assumption. This requires a proper tuning
between the total polynomial degree $k$ and the integer $L$ that 
determines the size of the partition. In the study of local polynomial
approximation, we treat separately the inner rectangles for which
$\ell\in \{0,\cdots,L-1\}^d$ and the infinite rectangles for which
one or several $\ell_j$ are equal to $L$. The estimates obtained
in the latter case rely on the additional assumption that the partition has
a geometry of disjoint inclusions.

\subsection{Polynomial approximation on inner rectangles}

Inner rectangles $R_\ell$ are particular cases of rectangles of the form
\be
\label{genrec}
R=[a_1,2a_1]\times \cdots \times [a_d,2a_d],
\ee
for some $a_j\geq 1$. The following lemma, adapted from \cite{BC}, shows that one can approximate the parameter to solution map in the $\|\cdot\|_y$ and $\|\cdot\|_{H^1_0}$ norms on such rectangles,
with a rate that decreases exponentially in the total polynomial degree.

\begin{lemma}
\label{lemmapolR}
Let $R$ be any rectangle of the form \iref{genrec}. Then, for each
$k\geq 0$, there exists functions $u_\nu\in H^1_0(\Omega)$ such that
\be
\Big\|u(y)-\sum_{|\nu|\leq k} u_\nu y^\nu\Big\|_{y}\leq C3^{-k}, \quad y\in R,
\label{yest}
\ee
where $C:=\frac 1 {\sqrt 3} C_f$, and
\be
\Big\|u(y)-\sum_{|\nu|\leq k} u_\nu y^\nu\Big\|_{H^1_0}\leq C3^{-k}, \quad y\in R,
\label{h10est}
\ee
where $C:=\frac 1 {\sqrt 6} C_f$.
\end{lemma}

\begin{proof}  \ac{The exponential rate is established in \cite{BC} for a single parameter domain
with uniform ellipticity assumption. Here the difficulty lies in the fact that we want the
same estimate for all parametric rectangles $R$ and thus without control on the uniform ellipticity.
Still the technique of proof, based on power series, is similar.}

The elliptic equation $-{\rm div}(a(y)u(y))=f$ may be written in operator form 
\[
A_y u(y)=f,
\]
where the invertible operator $A_y: H^1_0(\Omega)\to H^{-1}(\Omega)$ is defined
by 
\[
\<A_yv,w\>_{H^{-1},H^1_0}:=\int a(y)\nabla v\cdot\nabla w\, dx=\<v,w\>_y.
\]
We introduce 
\[
\o y:=\frac 3 2(a_1,\dots,a_d),
\]
the center of the rectangle, and write any $y\in R$ as
\[
y=\o y+\t y, 
\]
where the components $\t y_j$ of $\t y$ vary in $[-a_j/2,a_j/2]$.
We may write $A_y=A_{\o y}+\sum_{j=1}^d \t y_j A_j$, where the operators 
$A_j: H^1_0(\Omega)\to H^{-1}(\Omega)$
are defined by
\[
\<A_jv,w\>_{H^{-1},H^1_0}:=\int_{\Omega_j}\nabla v\cdot\nabla w\, dx.
\]
This allows us to rewrite the equation as
\[
(I+B(\t y)) u(y)=g, 
\]
where $g:=A_{\o y}^{-1}f\in H^1_0(\Omega)$ and 
$B(\t y)=\sum_{j=1}^d \t y_j A_{\o y}^{-1}A_j$ acts in $H^1_0(\Omega)$.
We then observe that 
\[
\<B(\t y)v,w\>_{\o y}=\<A_{\o y}B(\t y)v,w\>_{H^{-1},H^1_0}
=\sum_{j=1}^d \t y_j \<A_jv,w\>_{H^{-1},H^1_0}
=\sum_{j=1}^d \t y_j \int_{\Omega_j} \nabla v\cdot \nabla w\, dx,
\]
and therefore, since $|\t y_j|\leq \frac 1 3\o y_j$,
\[
|\<B(\t y)v,w\>_{\o y}| \leq \frac 1 3\sum_{j=1}^d \o y_j \Big |\int_{\Omega_j}\nabla v\cdot \nabla w\, dx\Big | \leq  \frac 1 3\|v\|_{\o y} \|w\|_{\o y},
\]
which shows that $\|B(\t y)\|_{\o y\to \o y} \leq \frac 1 3$. We may thus approximate
$(I+B(\t y))^{-1}$ by the partial Neumann series
\[
\sum_{l=0}^k (-1)^l B(\t y)^l, 
\]
which is a polynomial in $\t y$ of total degree $k$. The corresponding 
polynomial approximation to 
$u(y)$ is given by
\[
N_ku(y)=\sum_{l=0}^k (-1)^l B(\t y)^lg=
\sum_{l=0}^k (-1)^l \(\sum_{j=1}^d \t y_j A_{\o y}^{-1}A_j\)^l g
=\sum_{|\nu|\leq k} v_\nu \t y^\nu,
\]
and coincides with the truncated power series of $\t u(\t y):=u(\o y+\t y)$ at $\t y=0$,
that is, 
\[
v_\nu:=\frac 1 {\nu !} \partial^\nu u(\o y), \quad\quad \nu !:=\prod \nu_j!.
\]
It can 
be rewritten in the form 
\[
N_ku(y)=\sum_{|\nu|\leq k} u_\nu y^\nu.
\]
One has
\[
\|u(y)-N_ku(y)\|_{\o y} \leq \sum_{l> k} \|B(\t y)^l g\|_{\o y}
\leq \Big(\sum_{l> k} 3^{-l}\Big)\|A_{\o y}^{-1} f\|_{\o y} = \frac {3^{-k}} 2 \|A_{\o y}^{-1} f\|_{\o y},
\]
and 
\[
\|A_{\o y}^{-1} f\|_{\o y}^2=\<A_{\o y}A_{\o y}^{-1} f,A_{\o y}^{-1} f\>_{H^{-1},H^1_0}
=\<f,u(\o y)\>_{H^{-1},H^1_0} \leq C_f \|u(\o y)\|_{H^1_0}\leq  C_f^2,
\]
where the last inequality follows from Lax-Milgram estimate since
$a(\o y)\geq 1$. This proves the estimate
\be
\Big\|u(y)-\sum_{|\nu|\leq k} u_\nu y^\nu\Big\|_{\o y}\leq C3^{-k}, \quad y\in R,
\label{oyest}
\ee
with $C:=\frac 1 2 C_f$. Using the inequalities
\[
\|v\|_y^2 \leq \frac 4 3\|v\|_{\o y}^2, \quad v\in H^1_0(\Omega),\; y\in R,
\]
and
\[
\|v\|_{H^1_0}^2 \leq \frac 2 3\|v\|_{\o y}^2, \quad v\in H^1_0(\Omega),
\]
we obtain the estimate \iref{yest} and \iref{h10est} with the modified multiplicative constants.
\end{proof}

\begin{remark}
\label{remdim}
The above lemma shows that the set
$\cM_R:=\{u(y) \; : \; y\in R\}$ can be approximated
with accuracy $C3^{-k}$ by the space
\be
V_R:={\rm span}\{u_\nu \; : \; |\nu|\leq k\}.
\ee
The dimension of $V_R$ is at most ${k+d\choose d}$,
however, as noticed in \cite{BC},
it can in fact be seen that 
\be
\label{dimVR}
\dim(V_R)\leq {k+d-1\choose d-1}.
\ee
This stems from the fact that the operators defined in the above proof satisfy the dependency relation
\[
A_{\o y} =\sum_{j=1}^d \o y_j A_{j},
\]
and therefore, one can rewrite $A_y$ as
\[
A_y:=(1+\t y_d/\o y_d) A_{\o y}+\sum_{j=1}^{d-1}( \t y_j-\t y_d\o y_j/\o y_d  )A_{j}.
\]
Using this form, the partial Neumann sum $N_ku(y)$ has at most
${k+d-1\choose d-1}$ independent terms.
\end{remark}

We shall also make use of the following adaptation of the above lemma
to the approximation of the limit solution map $y_{S^c}\mapsto u_S(y_{S^c})$,
defined by \iref{limitproblem}. Its proof is an immediate adaptation of the previous one and is therefore omitted.

\begin{lemma}
\label{lemmapolRgen}
Let $S\subset \{1,\dots,d\}$, and for some
$a_j\geq 1$, let $R$ be a rectangle of the form
\be
R=\prod_{j\in S^c} [a_j,2a_j].
\ee
Then, there exists functions
$u_\nu\in V_S$ such that
\be
\Big\|u_S(y_{S^c})-\sum_{|\nu|\leq k} u_\nu y_{S^c}^\nu\Big\|_{y_{S^c}}\leq C3^{-k}, \quad y_{S^c}\in R,
\label{yestS}
\ee
where $C:=\frac 1 {\sqrt 3} C_f$, and
\be
\Big\|u_S(y_{S^c})-\sum_{|\nu|\leq k} u_\nu y_{S^c}\Big\|_{H^1_0}\leq C3^{-k}, \quad y_{S^c}\in R,
\label{h10estS}
\ee
where $C:=\frac 1 {\sqrt 6} C_f$.
\end{lemma}

\subsection{Polynomial approximation on infinite rectangles}

We now consider the infinite rectangles $R_{\ell}$, corresponding
to the $\ell$ such that some of the $\ell_j$ equal $L$. We define 
\be
S:=\{j\; : \; \ell_j=L\},
\label{Sell}
\ee 
the set of such indices. When $y\in R_\ell$, we thus have
\[
y_j\geq 2^L, \quad j\in S,
\]
and so $u(y)$ should be close to $u_S(y_{S^c})$ as $L$ is large.
On the other hand $y_{S^c}$ belongs to a rectangle of the form
\[
R_{\ell_{S^c}}=\prod_{j\in S^c} [2^{\ell_j},2^{\ell_j+1}].
\]
Therefore, by Lemma \ref{lemmapolRgen}, we can approximate
$u_S(y_{S^c})$ by a polynomial of total degree $k$ in these
restricted variables.

In order to conclude that this polynomial is a good approximation to
$u(y)$ on $R_{\ell}$, we need a quantitative estimate
on the convergence of $u(y)$ towards $u_S(y_{S^c})$. Let us observe
that since
\[
\sum_{j=1}^dy_j\int_{\Omega_j} \nabla u(y)\cdot\nabla v\, dx=\<f,v\>_{H^{-1},H^1_0}=\sum_{j\in S^c}y_j\int_{\Omega_j}\nabla u_S(y_{S^c})\cdot\nabla v\, dx,\quad v\in V_S,
\]
the function $u_S(y_{S^c})$ coincides with the orthogonal projection of $u(y)$
onto $V_S$ for the $y$-norm, as well as for the $y_{S^c}$-norm:
\be
\label{projuS}
u_S(y_{S^c})=P_{V_S}^y u(y)=P_{V_S}^{y_{S^c}} u(y).
\ee
In addition, with 
\be
\ac{\Omega_S:=\bigcup_{j\in S}\,\Omega_j,}
\label{OmegaS}
\ee 
we have
\[
2^L \|\nabla u(y)\|_{L^2(\Omega_S)}^2
\leq \sum_{j\in S} y_j\int_{\Omega_j} |\nabla u(y)|^2\, dx\leq \<f,u(y)\>_{H^{-1},H^1_0}
\leq C_f^2,
\]
since $\|u(y)\|_{H^1_0}\leq C_f$, and therefore, since $\nabla u_S(y_{S^c})=0$ on $\Omega_S$, we find that
\be
\|\nabla u(y)-\nabla u_S(y_{S^c})\|_{L^2(\Omega_S)}\leq C_f2^{-L/2}.
\label{estimonS}
\ee
Our objective is to obtain a similar error bound on the remaining domains
$\Omega_j$ for $j\in S^c$. \ac{This turns out to be feasible, with an even better rate $2^{-L}$, when making certain geometric assumptions on the partition of the domain $\Omega$.}

\begin{definition}
\ac{We say that $\{\Omega_1,\dots,\Omega_d\}$ is a {\em Lipschitz partition}
if and only if for any subset $T\subset \{1,\dots,d\}$, the domain $\Omega_T=\bigcup_{j\in T}\Omega_j$ has Lipschitz boundaries.}
\end{definition}

\begin{figure}[ht]
\begin{center}
\begin{minipage}{.2\textwidth}
\begin{tikzpicture}[scale=0.5]
  \draw(0,0)--(0,5);
  \draw(5,0)--(5,5);
  \draw(0,5)--(5,5);
  \draw(0,0)--(5,0);
  \node at (2,3.5) {$\Omega_1$};
  \draw (3,0) arc (0:90:3);
  \node at (4,1.5) {$\Omega_4$};
  \draw(3,5)--(3.5,3)--(5,3);
  \node at (4.25,4) {$\Omega_2$};
  \draw(3.5,3)--(2.598,1.5);
  \node at (1.2,1.2) {$\Omega_3$};
\end{tikzpicture}
\end{minipage}
\hspace{1cm}
\begin{minipage}{.2\textwidth}
\begin{tikzpicture}[scale=0.5]
  \draw(0,0)--(0,5);
  \draw(5,0)--(5,5);
  \draw(0,5)--(5,5);
  \draw(0,0)--(5,0);
  \draw(2.5,0)--(2.5,5);
  \draw(0,2.5)--(5,2.5);
  \node at (1.25,3.75) {$\Omega_1$};
  \node at (3.75,3.75) {$\Omega_2$};
  \node at (1.25,1.25) {$\Omega_3$};
  \node at (3.75,1.25) {$\Omega_4$};
\end{tikzpicture}
\end{minipage}
\end{center}
\caption{A Lipschitz partition of $\Omega$ (left) and a counter-example (right) since $\Omega_1\cup\Omega_4$ is not Lipschitz.}
\label{fig2}
\end{figure}
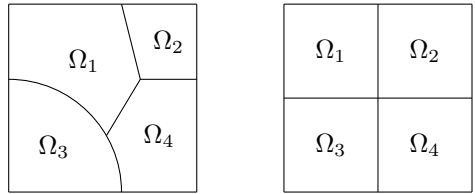

\ac{Note that such a property is stronger than just saying that each domain is Lipschitz, see Figure \ref{fig2} (right) for
a counter-example. In a Lipschitz partition, all subdomains $\Omega_j$ are Lipschitz, and the common 
boundary between two subdomains is either empty or a $(n-1)$-dimensional surface, as illustrated on Figure \ref{fig2} (left). 
In particular, it is easily checked that partitions consisting of a background domain
and well separated subdomains that have Lipschitz boundaries fall in this category. Similar to the $\Omega_T$,
the individual $\Omega_j$ could have several connected components, that should then be well separated.
Here by ``well separated'', we mean that $\delta$-neighbourhoods of the subdomains remain disjoints for some $\delta>0$.}

\ac{For the inner domains $\Omega_T$ such that $\partial\Omega_T\cap \partial\Omega=\emptyset$, the classical
Stein's extension theorem \cite{St} guarantees the existence of continuous extension operators 
\[
E_T: H^1(\Omega_T)\to H^1(\Omega),
\]
that satisfy $(E_T v)_{|\Omega_T}=v$ for all $v\in H^1(\Omega_T)$. We refer to chapter 5 of \cite{A} for a relatively simple construction of the extension operator $E_j$ by local reflection after
using a partitioning of unity along the boundary of $\Omega_T$ and local 
transformations mapping the boundary to the hyperplane $\R^{n-1}$.}

\ac
{For the domains $\Omega_T$ touching the boundary $\partial \Omega$, these operators are modified in order to take 
into account the homogeneous boundary condition, and we refer to \cite{Zen} for such adaptations. Here, the
relevant space is 
\be
\tilde H^1(\Omega_T):=R_T(H^1_0(\Omega)),
\label{tildeH}
\ee
where $R_T$ is the restriction to $\Omega_T$, over which $v\mapsto \|\nabla v\|_{L^2(\Omega_T)}$ is equivalent to 
the $H^1$ norm by Poincar\'e inequality. Then, there exists a continuous extension operator
\[
E_T: \t H^1(\Omega_T)\to H^1_0(\Omega).
\]
Note that the norm of all these operators
depends on the geometry of the partition.} These operators 
are instrumental in proving the following convergence estimate.

\begin{lemma}
\label{lemmaestimlim}
Assume that $\{\Omega_1,\dots,\Omega_d\}$ \ac{is a Lipschitz partition} of $\Omega$.
Then there exists a constant $C_0$ that only depends \ac{on the 
geometry of the partition} such that for any $S\subset \{1,\dots,d\}$ and
$y=(y_S,y_{S^c})\in Y'$, one has
\be
\label{estimconvy}
\ac{\|u(y)-u_S(y_{S^c})\|_{H^1_0}\leq C_0C_f \max_{j\in S} y_j^{-1}.}
\ee
In particular, for the infinite rectangle $R_{\ell}$,
\be
\ac{\|u(y)-u_S(y_{S^c})\|_{H^1_0}\leq C_0C_f 2^{-L}, \quad y\in R_{\ell},}
\label{estimconvL}
\ee
with $S$ defined by \iref{Sell}.
\end{lemma}

\begin{proof} \ac{We first note that it suffices to prove \iref{estimconvy} in the
particular case where the largest $y_j$ are those for which $j\in S$. 
Indeed, if this is not the case,
we use the decomposition
\[
u(y)-u_S(y_{S^c})=(u(y)-u_{S'}(y_{S'^c}))-(u(y')-u_{S'}(y_{S'^c}))+(u(y')-u_S(y_{S^c})),
\]
with $S'=\{i\; :\; y_i\geq \min_{j\in S}y_j\}$ and $y'$ defined by $y'_j=\max_{i=1,\dots,d} y_i$ if $j\in S$, $y'_j=y_j$ otherwise,
so that each term falls in this particular case and will be bounded in $H^1_0$ norm
by $C_0C_f \max_{j\in S} y_j^{-1}$. This leads to the same estimate \iref{estimconvy} up to a factor $3$ in constant $C_0$.
In addition, up to reordering the subdomains $\Omega_j$, we may assume $y_1\geq\dots\geq y_d$
and therefore $S=\{1,\dots,|S|\}$.
}

\ac{Fix $j\geq |S|$, and denote $u=u(y)$ and $u_S=u_S(y_{S^c})$ for simplicity. We define the Lipschitz domain
$\Omega^j=\o\Omega_1\cup\dots\cup\o\Omega_j$, remarking that 
\[
\Omega_S=\bigcup_{j\in S}\Omega_j=\Omega^{|S|}.
\]
Poincar\'e's inequality ensures that there exists a function $c$ on $\Omega^j$, constant on any connected component of $\Omega^j$, and null on $\partial\Omega\cap\Omega^j$, such that
\[
\|u-u_S-c\|_{H^1(\Omega^j)}\leq C_P\|\nabla (u-u_S)\|_{L^2(\Omega^j)},
\]
with $C_P$ the maximal Poincar\'e constant of all unions of subdomains from the partition. Moreover, there is an extension $v\in H^1_0(\Omega)$ of $u-u_S-c\in \t H^1(\Omega^j)$ such that
\[
\|v\|_{H^1_0(\Omega)}\leq C_E \|u-u_S-c\|_{H^1(\Omega^j)}\leq C_EC_P\|\nabla (u-u_S)\|_{L^2(\Omega^j)},
\]
with $C_E$ the maximal norm of all extension operators $E_{T}$, $T\subset \{1,\dots,d\}$.}

\ac{As $u-u_S-v=c$ on $\Omega_S\subset \Omega^j$, the function $u-u_S-v$ is in $V_S$, and therefore orthogonal to $u-u_S=u-P_{V_S}^{y}u$ for the $\|\cdot\|_y$ norm:
\begin{align*}
0&=\<u-u_S,u-u_S-v\>_y\\
&=\sum_{i=1}^d y_i\int_{\Omega_i}|\nabla (u-u_S)|^2-\sum_{i=1}^d y_i\int_{\Omega_i}\nabla (u-u_S)\cdot \nabla v\\
&=\sum_{i>j}y_i\int_{\Omega_i}|\nabla (u-u_S)|^2-\sum_{i>j} y_i\int_{\Omega_i}\nabla (u-u_S)\cdot \nabla v
\end{align*}
since $\nabla v= \nabla (u-u_S)$ on $\Omega^j$. In particular, we obtain
\begin{align*}
y_{j+1}\|\nabla (u-u_S)\|_{L^2(\Omega_{j+1})}^2 &\leq \sum_{i>j}y_i\int_{\Omega_i}|\nabla (u-u_S)|^2\\
&\leq y_{j+1}\int_{\Omega\setminus \Omega^j}|\nabla (u-u_S)\cdot \nabla v| \\
&\leq y_{j+1}\|u-u_S\|_{H^1_0(\Omega)}\|v\|_{H^1_0(\Omega)}\\
&\leq y_{j+1}\|u-u_S\|_{H^1_0(\Omega)}C_PC_E\|\nabla (u-u_S)\|_{L^2(\Omega^j)},
\end{align*}
and therefore
\[
\|\nabla (u-u_S)\|_{L^2(\Omega^{j+1})}^2\leq (1+C_PC_E)\|\nabla(u-u_S)\|_{L^2(\Omega)}\|\nabla (u-u_S)\|_{L^2(\Omega^j)}.
\]
Applying this inequality inductively for $j=d-1,\dots,d-k$, we get
\[
\|\nabla (u-u_S)\|_{L^2(\Omega)}\leq (1+C_PC_E)^{2^k-1} \|\nabla (u-u_S)\|_{L^2(\Omega^{d-k})},
\]
for any $k=1,\dots,d-|S|$. For $k=d-|S|$, this results in the bound
\be
\|\nabla (u-u_S)\|_{L^2(\Omega)}^2\leq C_0 \|\nabla (u-u_S)\|_{L^2(\Omega_S)}^2
=C_0 \|\nabla u\|_{L^2(\Omega_S)}^2,
\ee
for any non-empty $S$, with $C_0=(1+C_PC_E)^{2^{d-1}}$.}

\ac{We now write
\begin{align*}
(\min_{i\in S}y_i)\|\nabla (u-u_S)\|_{L^2(\Omega_S)}^2&\leq \|u-u_S\|_y^2=\<u,u-2u_S\>_y+\<u_S,u_S\>_{y_{S^c}}\\
&=\<f,u-u_S\>_{H^{-1},H^1_0} \leq C_f \|\nabla(u-u_S)\|_{L^2(\Omega)},
\end{align*}
which, combined to the previous estimate, gives
\[
\|u-u_S\|_{H^1_0}=\|\nabla (u-u_S)\|_{L^2(\Omega)}\leq C_0C_f \max_{i\in S}y_i^{-1},
\]
therefore proving \iref{estimconvy}.
For \iref{estimconvL}, we simply notice that $\max_{j\in S} y_j^{-1}\leq 2^{-L}$ for $y\in Y'\cap R_\ell$, and use a continuity argument when $y$ takes infinite values.}
\end{proof}

Combining the estimate \iref{estimconvL} from the above lemma with 
\iref{h10estS} from Lemma \ref{lemmapolRgen}, we obtain the following estimate
for polynomial approximation on an infinite rectangle $R_{\ell}$:
\be
\Big\|u(y)-\sum_{|\nu|\leq k} u_\nu y_{S^c}^\nu\Big\|_{H^1_0}\leq \frac{C_f}{\sqrt 6} \,3^{-k}+\ac{C_0C_f2^{-L},} \quad 
y\in R_\ell,
\label{h10polS}
\ee
where $C_0$ is the constant in \iref{estimconvL}. This estimate hints 
how the level $L$ in the partition should be tuned to the total polynomial degree $k$,
so that the two contributions in the above estimate are of the same order. 

\begin{remark}
\ac{Note that the constant $C_0=(1+C_PC_E)^{2^{d-1}}$ becomes prohibitive even for moderate values of $d$. However, under more restrictive geometric assumptions, for instance if the subdomains $\o \Omega_2,\dots,\o\Omega_d$ are disjoint inclusions in a background $\Omega_1$, better bounds can be obtained, with a constant $C_0$ that does not suffer a similar curse of dimensionality, by replacing the induction in the proof by a two-step procedure, consisting of extensions first from the high-diffusivity inclusions to the background, and then to the whole domain $\Omega$.
}
\end{remark}

\subsection{Approximation rates and $n$-widths}
\label{subsection 3 3}

We are now in position to establish an approximation result for the 
reduced model spaces. For this purpose, we fix the smallest level $L=L_k\geq 1$ such
that 
\[
\ac{C_0C_f2^{-L}}\leq \frac{C_f}{\sqrt 3} \,3^{-k}.
\]
In particular $L$ scales linearly with $k$, with the bound $\alpha k+\beta\leq L_k\leq \alpha k+\gamma$, where
\be
\label{estimL}
\alpha:=\frac {\ln 3}{\ln 2}, \quad \beta :=\frac{\ln(\sqrt 3 C_0)}{\ln 2}, \quad \gamma :=\frac{\ln(2\sqrt 3C_0)}{\ln 2}.
\ee
Then, the polynomial approximation estimates
\iref{h10est} and \iref{h10polS} show that for each
$\ell\in \{0,\cdots,L_k\}^d$, there exist functions 
$u_{\ell,\nu}\in H^1_0(\Omega)$ such that
\[
\Big\|u(y)-\sum_{|\nu|\leq k} u_{\ell,\nu} y^\nu\Big\|_{H^1_0}\leq \left(\frac{C_f}{\sqrt 6} +\frac{C_f}{\sqrt 3} \right)3^{-k}\leq C_f 3^{-k}, \quad  y\in R_\ell.
\]
Note that in the case of an infinite rectangle $R_\ell$, the $u_{\ell,\nu}$
are non trivial only for monomials of the form $y_{S^c}^\nu$ and they 
belong to $V_S$, where $S:=\{j\; : \; \ell_j=L_k\}$.

Thus the solutions $u(y)$ for $y\in R_\ell$ are approximated
with accuracy $C_f 3^{-k}$ in the space
\[
V_{\ell,k}:={\rm span}\{u_{\ell,\nu}\; : \; |\nu|\leq k\},
\]
which in view of Remark \ref{remdim} has dimension at most 
${k+d-1\choose d-1}$.

Note also that approximating the reduced manifold $\cN$ defined in \iref{normmanifold} 
requires a smaller subset of rectangles, since
\[
\{y\in \wt Y' \;: \;\min y_j=1\}\subset \bigcup_{\ell\in  E_k} R_\ell,
\quad  \quad E_k:=\{0,\cdots,L_k\}^d\setminus  \{1,\cdots,L_k\}^d.
\]
We thus introduce the reduced model space
\be
\label{Vn}
V_n:=\bigoplus_{\ell\in E_k} V_{\ell,k}, \quad n=\dim(V_n)\leq \#(E_k) {k+d-1\choose d-1},
\ee
and find that
\be
\label{approxk}
\|u(y)-P_{V_n}u(y)\|_{H^1_0}\leq C_f 3^{-k},
\ee
for all $y\in \wt Y'$ such that $\min y_j=1$. In view of \iref{estimL},
there exists a constant $C$ that depends on $d$ and $C_0$, such that
\be
n\leq ((L_k+1)^d-L_k^d){k+d-1\choose d-1}\leq C(k+1)^{2d-2}.
\label{estimn}
\ee
This leads to the following approximation theorem.

\begin{theorem}
\label{width}
Assume that the partition has the geometry of disjoint inclusions.
The reduced basis space $V_n$ defined in \iref{Vn} then satisfies
\be
\|u(y)-P_{V_n}u(y)\|_{H^1_0}
\leq C\exp\(-cn^{\frac 1 {2d-2}}\),
\label{estimate2d}
\ee
for all $y\in \wt Y'=[1,\infty]^d$ such that $\min y_j=1$. The Kolmogorov
$n$-width \iref{nwidth} of the reduced manifold $\cN$ satisfies
\be
d_n(\cN)_{H^1_0} \leq C\exp\(-cn^{\frac 1 {2d-2}}\).
\label{nw}
\ee
Over the full manifold $\o\cM$, one has the estimate in relative error
\be
\|u(y)-P_{V_n}u(y)\|_{H^1_0} \leq 
C\exp\(-cn^{\frac 1 {2d-2}}\)\|u(y)\|_{H^1_0},
\label{estimate2drel}
\ee
for all $y\in \wt Y=]0,\infty]^d$. The positive constants $c$ and $C$ 
only depend on $d$, $C_f$, and on the geometry of 
the partition through the constant $C_0$.
\end{theorem}

\begin{proof}
The estimate \iref{estimate2d} follows directly by combining
\iref{approxk} and \iref{estimn}, and \iref{nw} is an immediate consequence. We then derive \iref{estimate2drel} by using the homogeneity property \iref{homog}
and the lower inequality in \iref{frameh10}, similar to the proof
of \iref{releps} in Theorem \ref{theorelprojerror}.
\end{proof}

\begin{remark}
In the above construction of $V_n$, the dimension $n$ only takes the values
$n_k:=\#(E_k) {k+d-1\choose d-1}$ for $k\geq 0$. However it is easily seen
that if we set $V_n=V_{n_k}$ for $n_{k}\leq n< n_{k+1}$,
then all the estimates in the above theorem remain valid up to a change
in the constants $(c,C)$. 
\end{remark}

\begin{remark}
Note that the union of the $V_{\ell,k}$ for $\ell\in E_k$ would suffice
to approximate $\cN$ with uniform accuracy $C_f 3^{-k}$, their sum $V_n$
is an overkill. When $y$ is known, for example in forward modeling,
it is therefore possible to first identify the proper space $V_{\ell,k}$ associated
to the rectangle $R_{\ell}$ that contains $y$, and build the approximation
to $u(y)$ from this space. This nonlinear reduced modeling strategy
has been studied in \cite{BCDGJP} with similar local polynomial approximation under UEA, and in \cite{eftang2010,maday2013,zou2019adaptive} with local
reduced basis. The natural benchmark is given by the notion of library width
introduced in \cite{Tem}, that is defined for any compact set $\cK$ in a Banach
space $V$ as
\be
d_{n,N}(\cK)_V:=\inf_{\#(\cL_n)\leq N}\sup_{u\in\cK} \min_{V_n\in \cL_n}\min_{v\in V_n} \|u-v\|_V,
\ee
where the first infimum is taken over all libraries $\cL_n$ of $n$-dimensional
spaces with cardinality at most $N$. Our results thus show that
\[
d_{n,N}(\cN)_{H^1_0} \leq C_f3^{-k} \sim C\exp(-cn^{\frac 1 d}),
\quad n:={k+d-1\choose d-1}, \; N=(L_k+1)^d-L_k^d.
\]
\ac{Note that the above sub-exponential rate can be misleading due to fact that the constant $c$ 
has a hidden dependence in $d$. As an example, up to the constant $C_f$, we find that 
taking $k=4,7,9$ leads to error bounds $3^{-k}$ of order $10^{-2},10^{-3},10^{-4}$, with $n=15,36,55$ 
for $d=3$, and $n=35,120,220$ for $d=4$, which is far better than the value of $\exp(-n^{\frac 1 d})$.
}

\end{remark}

\begin{remark}
In view of the results from \cite{BCDDPW}
and \cite{DPW}, we are ensured that a proper selection of reduced basis
elements in the manifold $\cN$ should generate spaces $V_n$ that perform
at least with the same exponential rates as those achieved by the
spaces $V_n$ in Theorem \ref{width}. As explained in the introduction, 
reduced basis spaces may perform significantly
better than reduced model spaces based on polynomial or piecewise
polynomial approximation. \ac{This occurs in particular when the polynomial
coefficients have certain linear dependency, as established in \cite{BC}
for the elliptic problem with piecewise constant coefficients in the low contrast regime,
and recalled in Remark 3.2. There, it is shown that the rate ${\cal O}(\exp(-cn^{\frac 1 d}))$
is at least improved to ${\cal O}(\exp(-cn^{\frac 1 {d-1}}))$ and that further improvements in the rate
may result from certain symmetry properties of the domain partition, however not circumventing the
curse of dimensionality. While we do not pursue this analysis
in the present high contrast setting, we expect similar results to hold.
}
\label{remark reduced basis space}
\end{remark}

\section{Forward modeling and inverse problems}

\subsection{Galerkin projection}

In the context of forward modeling, the reduced model space $V_n$
is used to approximate the parameter to solution map, by 
a map
\[
y\mapsto u_n(y) \in V_n,
\]
computed through the Galerkin method: $u_n(y)\in V_n$ is such that
\[
\sum_{j=1}^dy_j\int_{\Omega_j} \nabla u_n(y)\cdot\nabla v\, dx=\<f,v\>_{H^{-1},H^1_0},\quad v\in V_n.
\]
Therefore $\<u_n(y),v\>_y=\<u(y),v\>_y$, that is 
\[
u_n(y)=P^y_{V_n}u(y),
\]
where $P^y_{V_n}$ is the projection onto $V_n$ with respect to norm $\|\cdot\|_y$.

Hence, one would like to derive estimates on $\|u(y)-P^y_{V_n}u(y)\|_{H^1_0}$
in place of the estimates on $\|u(y)-P_{V_n}u(y)\|_{H^1_0}$ that we have obtained
so far, since $P_{V_n}u(y)$ is not practically accessible. As explained
in the introduction, we cannot be satisfied with combining the latter estimates
with the bound
\[
\ac{\|u(y)-P_{V_n}^y u(y)\|_{H^1_0}\leq \kappa(y)^{1/2}\|u(y)-P_{V_n} u(y)\|_{H^1_0}}
\]
derived from Cea's lemma, since the multiplicative constant $\kappa(y)$ from \iref{contrast} is not uniformly bounded over the manifolds 
$\cM$, $\cB$ or $\cN$.  Here, we shall employ another approach to derive the same rates of convergence for $\|u(y)-P^y_{V_n}u(y)\|_{H^1_0}$.

One first observation is that in order for Galerkin projection 
$P_{V_n}^y$ onto a reduced model space $V_n$ to satisfy a
convergence bound in relative error, it is critical that 
this space contains some functions from the limit spaces
$V_S$. 
This is expressed by the following result.

\begin{proposition}
Assume that there exists $S\subsetneq\{1,\dots,d\}$ such that 
$V_n \cap V_S=\{0\}$. Then for any $C\in ]0,1[$, there exists $y\in Y'$ such that
\be
\label{nonrel}
\|u(y)-P^y_{V_n}u(y)\|_{H^1_0}\geq C\|u(y)\|_{H^1_0}.
\ee
\end{proposition}

\begin{proof} Since $V_n \cap V_S=\{0\}$, the quantity $\|\nabla v\|_{L^2(\Omega_S)}$ is a norm on $V_n$ and one can define
\[
\alpha=\min_{v\in V_n}\frac{\|\nabla v\|_{L^2(\Omega_S)}}{\|v\|_{H^1_0}}>0.
\]
For any $\eps>0$, take $y_j=\eps^{-2}$ for $j\in S$ and $y_j=1$ for $j\in S^c$. Then, for $v=P^y_{V_n}u(y)$,
\[
\frac\alpha{\eps}\|v\|_{H^1_0}\leq \frac{1}{\eps}\|\nabla v\|_{L^2(\Omega_S)}\leq \|v\|_y \leq  \|u(y)\|_y \leq C_f\leq \frac{C_f}{c_f}\|u(y)\|_{H^1_0},
\]
where we have used the framings \iref{frameh10} and \iref{framey}. Therefore,
taking $\eps=\frac{c_f}{C_f}\alpha(1-C)$ implies $\|v\|_{H^1_0}\leq (1-C)\|u(y)\|_{H^1_0}$, and \iref{nonrel} follows.
\end{proof}

However, in the construction of $V_n$ in \S 3, each space $V_{\ell,k}$ is a subset of $V_S$ for $S=\{j\;:\;\ell_j=L_k\}$. This prevents the phenomenon described
in the previous proposition from occurring. Instead, we obtain similar 
convergence bounds as those obtained for $P_{V_n}$, as expressed
in the following result.

\begin{theorem}
\label{widthy}
Assume that the partition of $\Omega$ has the geometry of disjoint inclusions.
On the rectangles $R_\ell$ for $\ell\in\{0,\dots,L\}^d$, the following
uniform convergence estimates hold:
\be
\|u(y)-P^y_{V_{\ell,k}}u(y)\|_{H^1_0}\leq \frac{C_f}{\sqrt 3} 3^{-k}, \quad y\in R_{\ell},
\label{estimellky}
\ee
if $\|\ell\|_\infty<L$, and 
\be
\|u(y)-P^y_{V_{\ell,k}}u(y)\|_{H^1_0}\leq  \frac{C_f}{\sqrt 3} 3^{-k}+C_0C_f2^{-L}, \quad y\in R_{\ell},
\label{estimellkyL}
\ee
if $\|\ell\|_\infty=L$. As a consequence, with $L=L_k$ 
and $V_n$ defined as in \S 3.3, one has the estimates
\be
\|u(y)-P_{V_n}^yu(y)\|_{H^1_0}
\leq C\exp\(-cn^{\frac 1 {2d-2}}\),
\label{estimate2dy}
\ee
for all $y\in \wt Y'$ such that $\min y_j=1$, and 
\be
\|u(y)-P^y_{V_{\ell,k}}u(y)\|_{H^1_0}
\leq C\exp\(-cn^{1/(2d-2)}\)\|u(y)\|_{H^1_0}, 
\label{estimate2dyrel}
\ee
for all $y\in\wt Y$, with constants $c$ and $C$ 
that only depend on $d$, $C_f$, and on the geometry of 
the partition through the constant $C_0$.
\end{theorem}

\begin{proof}
For bounded rectangles $R_\ell$ with $\|\ell\|_\infty < L$, we know from Lemma \ref{lemmapolR}, and more precisely from \iref{yest}, that
\[
\|u(y)-P_{V_{\ell,k}}^y u(y)\|_y=\min_{v\in V_{\ell,k}} \|u(y)-v\|_{y} \leq \Big\|u(y)-\sum_{|\nu|\leq k} u_\nu y^\nu\Big\|_{y}\leq \frac{C_f}{\sqrt 3}3^{-k}
\]
for any $y\in R_\ell$. Since all the $y_j$ are greater or equal to $1$, one has $\|v\|_{H^1_0}\leq \|v\|_y$ for all $v$
and therefore \iref{estimellky} follows.

For infinite rectangles $R_\ell$ such that $\|\ell\|_\infty=L$, we
again introduce $S=\{j\;:\;\ell_j=L\}$. Then, using \iref{estimconvL}, 

\begin{align*}
\|u(y)-P^y_{V_{\ell,k}}u(y)\|_{H^1_0}
&\leq \|u(y)-u_S(y_{S^c})\|_{H^1_0} + \|u_S(y_{S^c})-P^y_{V_{\ell,k}}u(y)\|_{H^1_0}\\
&\leq C_0 C_f2^{-L}+ \|u_S(y_{S^c})-P^y_{V_{\ell,k}}u(y)\|_{H^1_0}.
\end{align*}
Since $V_{\ell,k}\subset V_S$, we have
\[
P^y_{V_{\ell,k}}u(y)=P^y_{V_{\ell,k}}P^y_{V_S}u(y)=P^y_{V_{\ell,k}}u_S(y_{S^c})
=P^{y_{S^c}}_{V_{\ell,k}}u_S(y_{S^c}),
\]
Similarly to the previous case, we apply \iref{yestS} from Lemma~\ref{lemmapolRgen}:
\[
\|u_S(y_{S^c})-P_{V_{\ell,k}}^y u_S(y_{S^c})\|_{H^1_0} \leq \|u_S(y_{S^c})-P_{V_{\ell,k}}^y u_S(y_{S^c})\|_{y} \leq \frac{C_f}{\sqrt 3} 3^{-k},
\]
and we thus obtain \iref{estimellkyL}. 

After taking $L=L_k$ and defining $V_n$ as the sum of the $V_{\ell,k}$ for $\ell\in E_k$,
the derivation of \iref{estimate2dy} and \iref{estimate2dyrel} is exactly the same as for \iref{estimate2d} and \iref{estimate2drel}.
\end{proof}

\begin{remark}
As in Remark \ref{remark reduced basis space}, it is expected that the same rate of convergence is attained if $V_n$ is a reduced basis space generated by solutions $u(y^i)$, $i=1,\dots,n$, as long as there are $O\big({k+d-1 \choose d-1}\big)$ samples $y^i$ in each rectangle, 
however with samples forced to be of the form $u_S(y_{S^c}^i)\in V_S$ in the case of infinite rectangles.
\end{remark}

\subsection{State and parameter estimation}

The state estimation problem consists in retrieving the
solution $\o u=u(\o y)$ when the parameter $\o y$ is unknown,
and one observes $m$ linear measurements
\[
w_i=\ell_i(\o u), \quad i=1,\dots,m,
\]
where the $\ell_i$ are continuous linear functional on the Hilbert space
$V$ that contains the solution manifold. These linear functionals
may thus be written in terms of Riesz representers
\[
\ell_i(v)=\<\omega_i,v\>_V.
\]
The Parametrized Background Data Weak (PBDW) method,
introduced in \cite{MPPY} and further studied in \cite{BCDDPW17},
exploits the fact that all potential solutions are well approximated
by reduced model spaces $V_n$. It is based on
a simple recovery algorithm that consists in solving the problem
\be
\min_{u\in V_w}\min_{v\in V_n} \|u-v\|_V,
\label{doublemin}
\ee
where, for  $w=(w_1,\dots,w_m)\in \R^m$,
\[
V_w:=\{u\in V \; : \; \ell_i(u)=w_i, \;i=1,\dots,m\},
\]
is the affine space of functions that agree with the measurements. 

The analysis of this problem is governed by the quantity
\be
\mu_n=\mu(V_n,W):=\sup_{v\in V_n} \frac{\|v\|_V}{\|P_Ww\|_V},
\ee
where $W:={\rm span}\{\omega_1,\dots,\omega_m\}$, which is finite
if and only if $V_n\cap W^\perp=\{0\}$. Then, 
there exists a unique minimizing pair 
\[
(u^*,v^*)=(u^*(w),v^*(w))\in V_w\times V_n
\]
to \iref{doublemin}, which satisfies the estimates
\be
\|\o u-v^*\|_V \leq \mu_n\min_{v\in V_n}\|u-v\|_V,
\label{vstar}
\ee
and 
\be
\|\o u-u^*\|_V \leq \mu_n \min_{v\in V_n+(W\cap V_n^\perp)}\|u-v\|_V.
\label{ustar}
\ee
The computation of $(u^*,v^*)$ amounts to solving finite linear systems, and both solutions depend linearly on $w$.

Turning to our specific elliptic problem, and assuming that
the $\ell_i$ belong to $H^{-1}(\Omega)=V'$ for $V=H^1_0(\Omega)$, we may apply the
above PBDW method using the reduced basis spaces $V_n$
introduced in \S 3. As an immediate consequence of Theorem \iref{width}, we obtain a recovery estimate in relative error.

\begin{proposition}
Let $\o y\in\wt Y$ and $\o u=u(\o y)$. Then \ac{ 
both 
estimators $v^*\in V_n$ and $u^*\in V_w$ satisfy}
\be
\ac{\max\{\|\o u- v^*\|_{H^1_0},\|\o u-u^*\|_{H^1_0}\} \leq 
C\mu_n\exp\(-cn^{\frac 1 {2d-2}}\)\|\o u\|_{H^1_0}.}
\label{estimate2drelstate}
\ee
The positive constants $c$ and $C$ 
only depend on $d$, $C_f$, and on the geometry of 
the partition through the constant $C_0$.
\end{proposition} 

\begin{proof}
It follows readily by combining \iref{estimate2drel}
applied to $y=\o y$ with the recovery estimates
\iref{vstar} and \iref{ustar}.
\end{proof}

We next turn to the problem of parameter estimation, namely 
recovering an approximation $y^*$ to $\o y$ from the measurements
$w$. In contrast to state estimation,
this is a nonlinear inverse problem since the first mapping in
\[
\o y\mapsto \o u \mapsto w
\]
is typically nonlinear. One way of relaxing this problem into a linear one is by first using a recovery $u^*$ of the state $\o u$, for example
obtained by the PBDW method. One then defines $y^*$ as the
minimizer over $\wt Y$ of the residual
\[
R(y):=\|{\rm div}(a(y)\nabla u^*)+f\|_{H^{-1}}.
\]
This is a quadratic problem when $a(y)$ has an affine dependence
in $y$, that can be solved by standard quadratic optimization methods. The rationale for this approach is the fact that
\[
R(y)=\|A_y u^*-A_y u(y)\|_{H^{-1}}\sim \|u^*-u(y)\|_{H^1_0},
\]
and therefore we should be close to finding the parameter $y$ that best explains
the approximation $u^*$. Unfortunately, this approach is not
much viable in the high-contrast regime since the 
equivalence $\|A_y v\|_{H^{-1}}\sim \|v\|_{H^1_0}$ has constants
that are not uniform in $y$ and deteriorate with the level of contrast.

Instead, we propose a more specific approach that 
exploits the piecewise constant structure of $a(y)$,
assuming that $V_n$ is a reduced space of the form
\[
V_n=\Span(u^1,\dots,u^n), \quad u^i=u(y^i),
\]
for some properly selected parameter vectors
\[
y^i=(y^i_1,\dots,y^i_d),\quad i=1,\dots,n.
\]
As mentioned, see Remark \iref{remark reduced basis space}, these spaces satisfy the same exponential convergence bounds as the spaces constructed in \S 3.

The PBDW estimator $v^*=v^*(w)\in V_n$ thus has the form
\[
v^*=\sum_{i=1}^n c_iu^i\in V_n
\]
and satisfies a similar bound \iref{estimate2drelstate} as in the above
proposition. Then, on the particular domain $\Omega_j$, one has
\[
\frac f{\o y_j}=-\Delta \o u_{|\Omega_j}\approx -\sum_{i=1}^n c_i\Delta u^i= \sum_{i=1}^n c_i\frac{f}{y^i_j},
\]
and therefore, a natural candidate for the parameter estimate
is $y^*=(y^*_1,\dots,y^*_d)$ with
\be
y_j^*:=\left(\sum_{i=1}^n \frac{c_i}{y^i_j}\right)^{-1}.
\label{ystar}
\ee
The following result gives a recovery bound in relative error for
the inverse diffusivity.

\begin{proposition}
With the notation $1/y=(1/y_1,\dots,1/y_d)$, the estimator $y^*$ defined by \iref{ystar} satisfies the bound
\be
\Big\|\frac{1}{y^*}-\frac{1}{\o y}\Big\|_{\infty}
\leq 
\frac {C_f }{c_f}C\mu_n\exp\(-cn^{\frac 1 {2d-2}}\)\Big\|\frac 1 {\o y}\Big\|_{\infty},
\label{yystar}
\ee
where $C_f$ and $c_f$ are as in \iref{frameh10}, and the other constants as in \iref{estimate2drelstate}. 
\end{proposition}

\begin{proof}
For $1\leq j \leq d$, take $\phi\in H^1_0(\Omega_j)$, then

\begin{align*}
\left|\frac{1}{ y_j^*}-\frac{1}{\o y_j}\right| |\<f,\phi\>_{H^{-1},H^1_0}|
&=\left|\sum_{i=1}^n\frac{c_i}{y^i_j}\int_{\Omega_j}y^i_j\nabla u^i\cdot\nabla \phi\, dx-\frac{1}{\o y_j}\int_{\Omega_j}\o y_j\nabla \o u\cdot\nabla \phi\, dx\right|\\
&=\left|\int_{\Omega_j}\nabla (v^*-\o u)\cdot\nabla \phi\, dx\right|\\
&\leq \|v^*-\o u\|_{H^1_0(\Omega)}\|\phi\|_{H^1_0(\Omega_j)}.
\end{align*}
Optimizing over $\phi$ gives
\[
\Big\|\frac{1}{y^*}-\frac{1}{\o y}\Big\|_{\infty}
\leq 
c_f^{-1}\|v^*-\o u\|_{H^1_0},
\]
which combined with \iref{estimate2drelstate} gives
\[
\Big\|\frac{1}{y^*}-\frac{1}{\o y}\Big\|_{\infty}
\leq 
c_f^{-1}C\mu_n\exp\(-cn^{\frac 1 {2d-2}}\)\|\o u\|_{H^1_0}.
\]
Using the Lax-Milgram estimate 
\[
\|\o u\|_{H^1_0}\leq C_f \Big\|\frac 1 {\o y}\Big\|_{\infty},
\]
we reach \iref{yystar}. 
\end{proof}

\begin{remark}
The bound \eqref{yystar} is not entirely satisfactory since the approximation error on $\o y_j$ remains high when $\o y\in \cN$ with $\o y_j\gg 1$. We do not know if a bound of the form
\[
\left|\frac{1}{y_j^*}-\frac{1}{\o y_j}\right| \leq \frac{\eps_n}{\o y_j},\quad 1\leq j \leq d, 
\]
which would imply $|y_j^*-\o y_j|\leq {\eps_n}/(1-\eps_n)\,\o y_j$, holds uniformly over $\cN$ with $\eps_n\underset{n\to+\infty}{\longrightarrow}0$.
\end{remark}

\section{Numerical illustration}

\ac{The base model that will be used all along the numerical illustrations is the diffusion equation \iref{ellip} with data $f=1$ 
set on the two-dimensional square $\Omega=[-1, 1]^2$
with homogeneous Dirichlet boundary conditions. We consider a piece-wise constant diffusion coefficient
\[
a_{|\Omega_j} = y_j, \quad 1\leq j \leq d,
\]
on a partition of $\Omega$ into $16$ squares of quarter side-length.}

\ac{As such this partition does not satisfy the geometrical assumption of ``Lipschitz partition'' that was critical
in our analysis for the application of Lemma \ref{lemmaestimlim}. 
Therefore we consider sub-partitions that comply to the assumptions, such as
illustrated on Figure \ref{fig:GeomAssumption}, which amounts to equate the parameters
$y_j$ of squares belonging to the same sub-domain. This way we can consider that $y = (y_A, y_B, y_C, y_D)$ consists of four parameters, one per each subdomain.}

\begin{figure}
\begin{center}

\begin{minipage}[c]{0.4\textwidth}
\begin{tikzpicture}[scale = .90]
  \draw(0,0)--(0,4);
  \draw(4,0)--(4,4);
  \draw(0,4)--(4,4);
  \draw(0,0)--(4,0);
  \node at (1.5, 1.5) {$\Omega_D$};
  
  \draw (1, 3) rectangle (4, 2);
  \node at (2.5, 2.5) {$\Omega_C$};
  
  \draw (3, 1) rectangle (4, 2);
  \node at (3.5, 1.5) {$\Omega_B$};
  
  \draw (1, 1) rectangle (2, 0);
  \node at (1.5,0.5) {$\Omega_A$};
\end{tikzpicture}
\caption{Lipschitz partition\\ of $\Omega$.}
\label{fig:GeomAssumption}
\end{minipage}
\hspace{1cm}
\begin{minipage}[c]{0.4\textwidth}
\begin{tikzpicture}[scale = .90]
  \draw[step=1cm, black, thin] (0,0) grid (4,4);
  \node at (0.5,0.5) {$\Omega_A$};
  \node at (1.5,1.5) {$\Omega_A$};
  \node at (2.5,2.5) {$\Omega_A$};
  \node at (3.5,3.5) {$\Omega_A$};
  
  \node at (0.5,2.5) {$\Omega_B$};
  \node at (1.5,3.5) {$\Omega_B$};
  \node at (2.5,0.5) {$\Omega_B$};
  \node at (3.5,1.5) {$\Omega_B$};
  
  \node at (1.5,0.5) {$\Omega_C$};
  \node at (0.5,1.5) {$\Omega_C$};
  \node at (3.5,2.5) {$\Omega_C$};
  \node at (2.5,3.5) {$\Omega_C$};
  
  \node at (0.5,3.5) {$\Omega_D$};
  \node at (1.5,2.5) {$\Omega_D$};
  \node at (2.5,1.5) {$\Omega_D$};
  \node at (3.5,0.5) {$\Omega_D$};
\end{tikzpicture}
\caption{Non-lipschitz\\  partition of $\Omega$.}
\label{fig:NoGeomAssumption}
\end{minipage}

\end{center}
\end{figure}

\ac{The numerical results that we next
present aim to illustrate the robustness to high-contrast of the reduced basis method, and
discuss in addition the effect of parameter selection, higher parametric dimensions, 
and inclusions that are not satisfying the geometric assumption as exemplified on Figure
\ref{fig:NoGeomAssumption}.}

\ac{We construct different reduced bases
$\{u^1,\dots,u^n\}$ of moderate dimension $\displaystyle{1\leq n\leq 15}$, where
\[
u^k=u(y^k),
\]
for certain parameter selections $y^1,\dots,y^n$. Each reduced basis
element $u^k$ is numerically computed by the Galerkin method 
in a background finite element space $V_h$ of dimension $6241$.}

\ac{The reduced basis spaces are thus subspaces of $V_h$, thus strictly speaking
spaces $V_{n,h}$ depending on $n$ and on the meshsize $h$. In our numerical computation, we 
always assess the error 
\[
P^y_{V_h}u(y)-P^y_{V_{n,h}}u(y).
\]
We noticed that for the considered values of $n=1,\dots,15$ the error curves do not vary much
when further reducing the mesh size $h$. In fact they are already essentially the same
when the dimension of $V_h$ is four times smaller. Therefore, for simplicity of the presentation, we still write
\[
u(y)-P^y_{V_{n}}u(y),
\]
bearing in mind that the additional finite element error $u(y)-P^y_{V_h}u(y)$ depends on $h$
(with algebraic decay in the finite element dimension).}

\ac{All the tests were done using Python 3.8. For more information and experiments not presented here we invite the reader to look into the github repository \url{https://github.com/agussomacal/ROMHighContrast}.}

\subsection{Parameter selection}

We first study the case of a one parameter family : the diffusion coefficient $y_A$ of $\Omega_A$ in Figure \ref{fig:GeomAssumption} 
varies from $1$ to $\infty$, while the other subdomains are considered as background with
all coefficents equal to $1$. Thus the $y^k$ are of the form $y^k=(y_A^k,1,1,1)$.

\begin{figure}[ht]
\centering
\hspace{-2cm}
\begin{minipage}{.45\textwidth}
\includegraphics[scale=0.35]{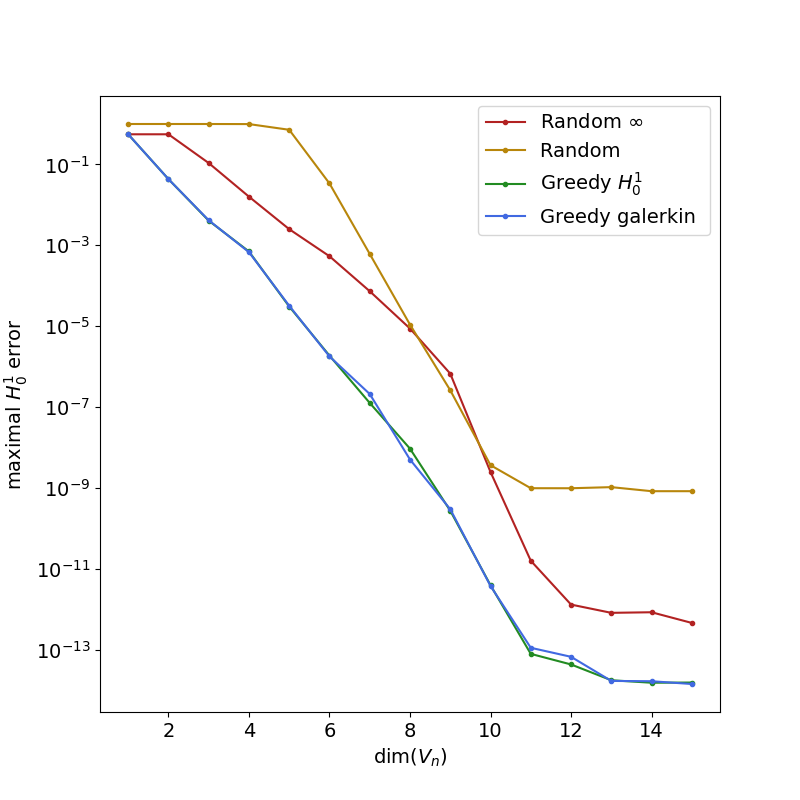}
\end{minipage}
\hspace{1cm}
\begin{minipage}{.45\textwidth}
\includegraphics[scale=0.35]{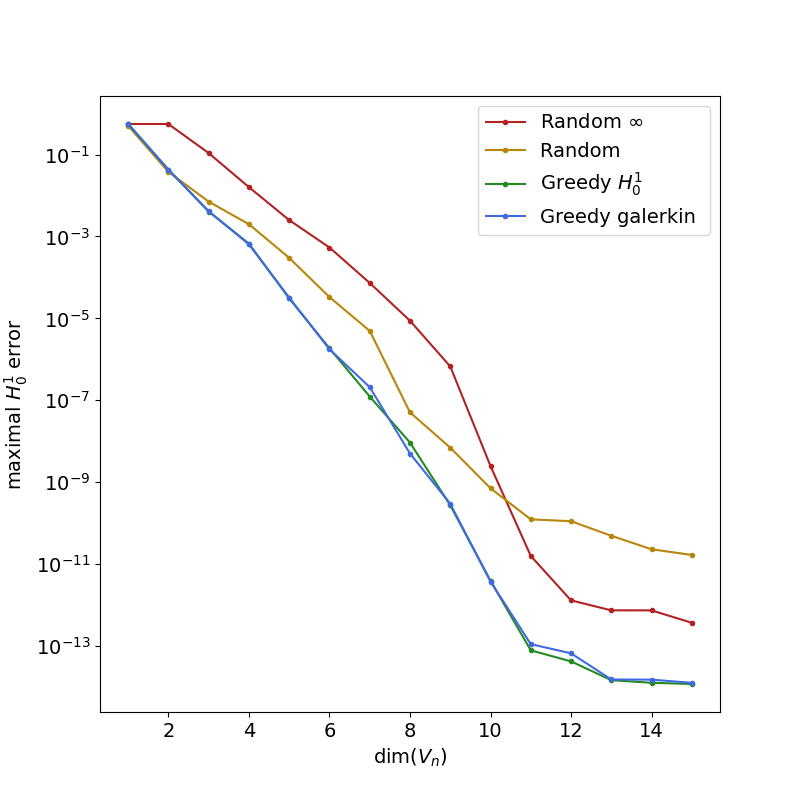}
\end{minipage}
\caption{Galerkin (left) and $H^1_0$ (right) projection error, both measured in $H^1_0$ relative error, maximized
over the parameter domain, for different reduced bases, case $d=1$.}
\label{fm_rates}
\end{figure}

\ac{In reduced basis constructions, two approaches for parameter selection are usually 
considered : random or greedy. Random selection usually performs well enough 
in many situations, however we shall see that it fails in the high contrast regime. This is 
in particular due to the fact that it does not capture the limit solutions,
while we have observed in \S 4 that robust convergence of the Galerkin method
in the high-contrast regime critically requires to include limit solutions
in the space $V_n$. Here, there is only one limit solution $u_\infty=u(y_\infty)$ where $y_\infty=(\infty,1,1,1)$,
and this element is picked by the greedy method if initialized at any other point. }

\ac{More precisely, we compare four strategies for selecting the $y_A^k\in [1,\infty]$:
\begin{itemize}
    \item Random: the $y_A^k$ are drawn independently according to the
uniform law for $\frac 1 {y_A} \in [0,1]$.
    \item Random-$\infty$: First the limit solution corresponding to $y_A=\infty$ is put
    in the basis. The rest of the elements are randomly picked as in the previous case.
    \item Greedy $H^1_0$: The $y^k$ are picked incrementally, $y^{k+1}$ maximizing the relative $H^1_0$ projection error $\|u(y)-P_{V_k} u(y)\|_{H^1_0}/\|u(y)\|_{H^1_0}$.
    \item Greedy Galerkin: The $y^k$ are picked incrementally, $y^{k+1}$ maximizing the relative $H^1_0$ error of the Galerkin projection 
    $\|u(y)-P_{V_k}^y u(y)\|_{H^1_0}/\|u(y)\|_{H^1_0}$.
\end{itemize}
}

\ac{Figure \ref{fm_rates} displays on the left the evolution of the maximal relative error of the Galerkin projection
\[
\sup_{y_A\in[1,\infty]} \frac{\|u(y)-P_{V_n}^yu(y)\|_{H^1_0}}{\|u(y)\|_{H^1_0}},
\]
as a function of $n=\dim(V_n)$ for these various selection strategies. It reveals the superiority of the greedy selection
that reaches machine precision after picking $n=11$ reduced basis elements, and the gain in including the limit solution 
in the case of a random selection.
As a comparison, we display on the right the decay of the relative $H^1_0$-orthogonal projection error 
\[
\sup_{y_A\in[1,\infty]} \frac{\|u(y)-P_{V_n} u(y)\|_{H^1_0}}{\|u(y)\|_{H^1_0}}
\]
for the same parameter selection strategies. Here, we notice that the inclusion of the limit solution $u_\infty$ is not anymore critical for reaching good accuracy. Nevertheless, these errors still decay faster for the greedy strategies.}

\begin{remark}
\ac{As the diffusion coefficient is piecewise constant on the partition $\Omega_A\cup \Omega_A^c$, the parameter space dimension is $d=2$ in this numerical example. The theoretical results thus provide a bound on the error of order $\exp(-c\sqrt n)$. However, this bound is obtained with local reduced spaces $V_{\ell,k}$ on dyadic intervals, which does not perform as well as $V_n=\bigoplus_{\ell\in E_k} V_{\ell,k}$, for which one might expect a rate closer to $\exp(-c n)$. In Figure \ref{fm_rates} for $n\leq 11$, that is, until numerical precision issues arise, we even observe a faster than exponential convergence, that could be due to the superiority of reduced bases over polynomial approximations.}
\end{remark}


%

\begin{remark} It is well known that the reduced basis can be very ill-conditioned, since $u^n$ becomes extremely close to $V_{n-1}={\rm span}\{u^1,\dots,u^{n-1}\}$ as
$n$ gets moderately large. In order to avoid numerical instabilities,
prior to the computation
of the Galerkin or $H^1_0$ projection onto $V_n$, we need
to perform a change of basis, typically by some orthonormalization process.
In our numerical test, we perform this orthonormalization with respect
to the discrete $\ell^2$ inner product for the nodal values in the 
background finite element representation, using the QR decomposition, and obtain a satisfactory stable numerical behavior. However, this process is not invariant under permutations, and we observe that it
behaves better in terms of numerical stability when sorting the reduced basis elements from higher contrast to lower contrast.
\end{remark}

\begin{figure}[ht]
\centering
\includegraphics[scale=0.35]{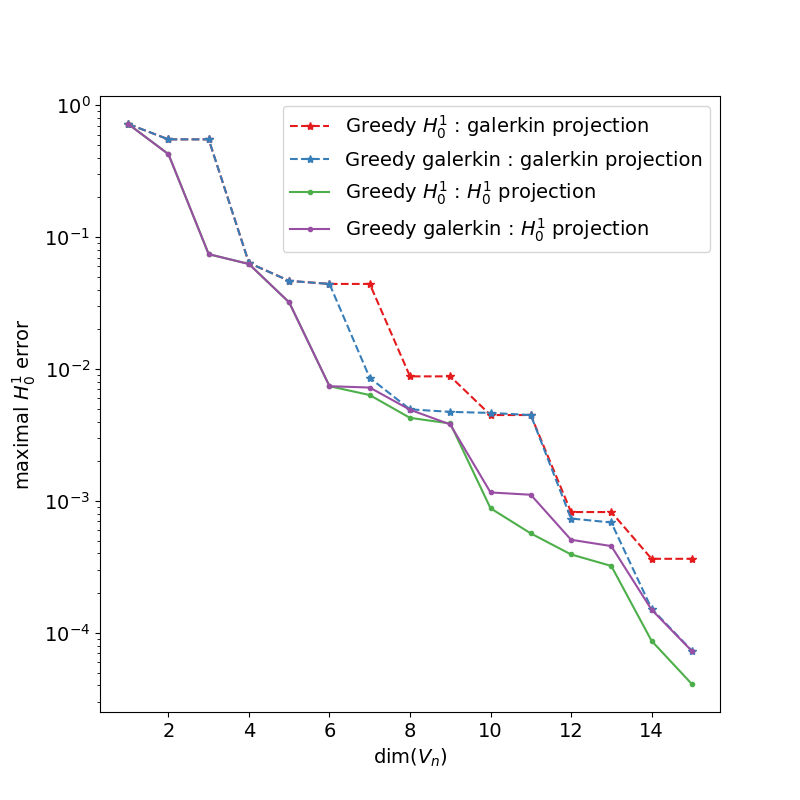}
\caption{Galerkin and $H^1_0$ projection error (both measured in $H^1_0$ relative error maximized
over the parameter domain) for different reduced bases, case $d=2$.}
\label{fm_rates_d2}
\end{figure}

\ac{In this one parameter scenario, both greedy strategies behaved equally well. However, as we increase the dimensionality of the problem $d>1$, Greedy Galerkin appears to be the best selection procedure, as could be expected since it optimizes the error based on the approximation which is effectively computed in forward modeling. Figure \ref{fm_rates_d2} shows this effect when $d=2$, where $y_A$ and $y_B$ are allowed to vary independently while $y_C$ and $y_D$ are taken as background always equal to $1$.}

\begin{figure}[ht]
\centering
\hspace{-2cm}
\begin{minipage}{.45\textwidth}
\includegraphics[scale=0.35]{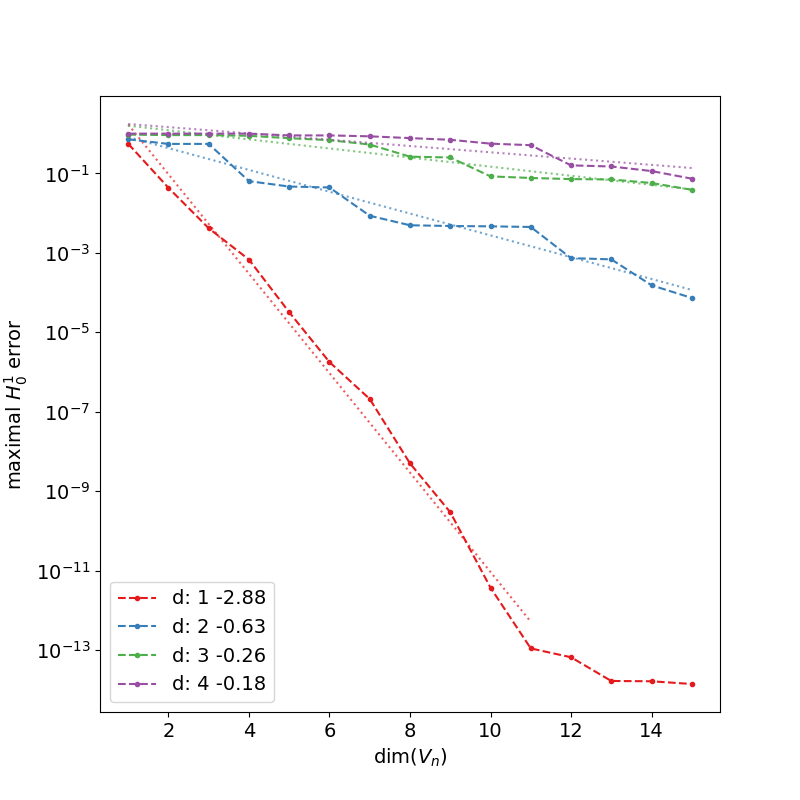}
\end{minipage}
\hspace{1cm}
\begin{minipage}{.45\textwidth}
\includegraphics[scale=0.35]{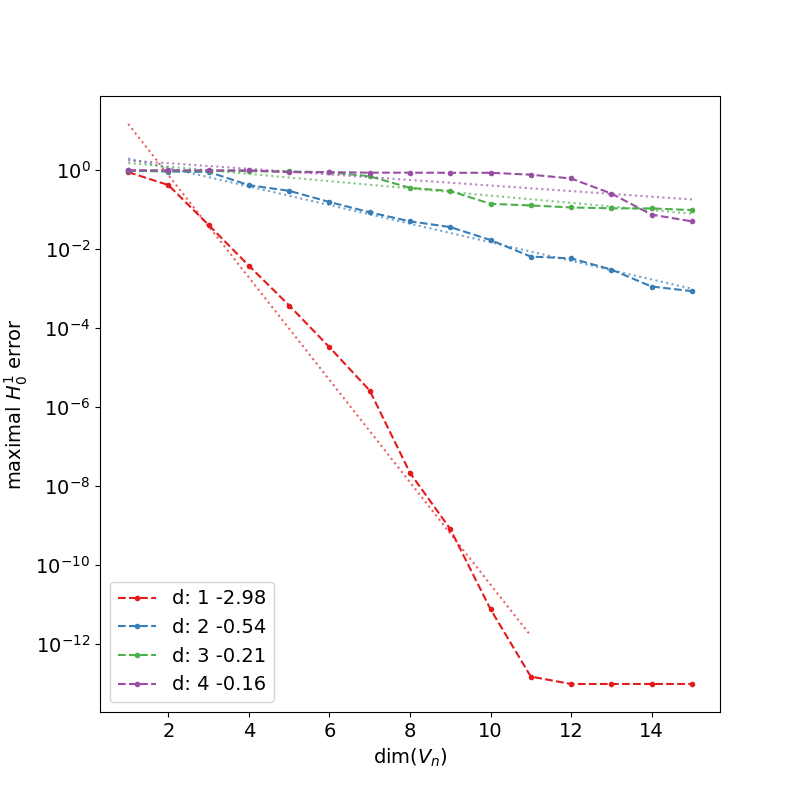}
\end{minipage}
\caption{The Galerkin projection of Greedy Galerkin method for increasing dimensionality in geometries satisfying (left) or not (right) the assumptions.}
\label{fig:dim_effect}
\end{figure}

\subsection{Influence of dimensionality and geometry}

\ac{In order to study the impact of dimensionality on the approximation rates,
we compare the behavior of the Greedy Galerkin selection method,
as we increase the number of freely varying parameters. As before, we will have 
for $y=(y_A,1,1,1)$ when $d=1$, then $y=(y_A,y_B,1,1)$ when $d=2$, 
until having all four subdomains freely varying between $1$ and $+\infty$. }

\ac{In Figure \ref{fig:dim_effect} the degradation with respect to dimension 
is clearly observed as the approximation capabilities strongly decrease. 
Even thought the  exponential decay rate is still conserved, the decay 
parameter shrinks from almost $3$ down to $0.22$ when $d=4$.  }


\ac{Secondly, we study the case where the geometrical assumptions are not satisfied. We follow the same incremental subdomains unfreezing as in the previous case but using the geometry stated in Figure \ref{fig:NoGeomAssumption}. We observe that the reduced basis approach still 
achieves exponential approximation rates, actually higher than in the previous example. This hints that the geometric assumptions
which are needed in our proofs could be artificial, and leaves open the question of achieving such results without relying on these assumptions.}



\begin{thebibliography}{10}

\bibitem{A} R. A. Adam and J. F. Fournier, {\it Sobolev spaces}, Elsevier, 2003.

\bibitem{AY}\ac{B. Aksoyly and Z. Yelter, {\it Robust multigrid preconditioners for cell-centered 
finite volume discretization of the high-contrast diffusion equation}, 
Computing and Visualization in Science 13, 229-245, 2010.}

\bibitem{AGKS}\ac{B. Aksoylu, I.G. Graham, H. Klie, and R. Scheichl, {\it Towards a rigorously justified algebraic preconditioner for
high-contrast diffusion problems}, Computing and Visualization 11, 319-331, 2008.}

\bibitem{Ain} \ac{M. Ainsworth, {\it Robust a posteriori error estimation for nonconforming finite element approximation}, 
SIAM J. Num. Anal 42-6, 2320-2341, 2005.}

\bibitem{BNT} I. Babu\v{s}ka, F. Nobile and R. Tempone, {\it A stochastic collocation method for elliptic partial differential equations with random input data}, SIAM J. Num. Anal. 45, 1005-1034, 2007.

\bibitem{BNTT1} J. Beck, F. Nobile, L. Tamellini and R. Tempone,
{\it Stochastic spectral Galerkin and collocation methods for PDEs with random coefficients: a numerical comparison}, Lecture Notes in Computational Science and Engineering 76, 43-62, 2010.

\bibitem{BNTT2} J. Beck, F. Nobile, L. Tamellini and R. Tempone,
{\it Implementation of optimal Galerkin and collocation approximations of PDEs with random coefficients}, ESAIM Proc 33, 10-21, 2011. 

\bibitem{BNTT3} J. Beck, F. Nobile, L. Tamellini and R. Tempone, {\it Convergence of quasi-optimal stochastic Galerkin methods for a class of PDES with random coefficients}, Computers \& Mathematics with Applications 67(9), 732-751, 2014.

\bibitem{BV}\ac{C. Bernardi and R. Verf\"urth, {\it Adaptive finite element methods for elliptic equations with non-smooth coefficients},
Numerische Mathematik 85, 579-608, 2000.}

\bibitem{BC} M.~Bachmayr and A.~Cohen.
\newblock {\it Kolmogorov widths and low-rank approximations of parametric elliptic PDEs},
\newblock Mathematics of Computation, vol. 86, no. 304, 701-724, 2017.

\bibitem{BCM}
M.~Bachmayr, A.~Cohen, and G.~Migliorati.
\newblock {\it Sparse polynomial approximation of parametric elliptic pdes. part i:
  affine coefficients},
\newblock ESAIM:M2AN, 51(1), 321-339, 2017.

\bibitem{BCDDPW}
P.~Binev, A.~Cohen, W.~Dahmen, R.~DeVore, G.~Petrova, and P.~Wojtaszczyk.
\newblock {\it Convergence rates for greedy algorithms in reduced basis methods},
\newblock SIAM Journal on Mathematical Analysis, 43(3), 1457-1472, 2011.

\bibitem{BCDDPW17}
P.~Binev, A.~Cohen, W.~Dahmen, R.~DeVore, G.~Petrova, and P.~Wojtaszczyk.
\newblock {\it Data assimilation in reduced modeling},
\newblock SIAM/ASA Journal on Uncertainty Quantification, 5(1), 1-29,
  2017.

\bibitem{BCDGJP} A. Bonito, A. Cohen, R. DeVore, D. Guignard, P. Jantsch and G. Petrova, {\it Nonlinear methods for model reduction}, ESAIM: Mathematical Modelling and Numerical Analysis, 55(2), 507-531, 2021.

\bibitem{BMPPT} A. Buffa, Y. Maday, A.T. Patera, C. Prud'homme, and G. Turinici,  {\it A Priori convergence of the greedy algorithm for the parameterized reduced basis},  Mathematical Modeling and Numerical Analysis, 46, 595-603, 2012.

\bibitem{Ch} \ac{A. Chatterjee, {\it An introduction to the proper orthogonal decomposition}, Current Science 78, 539-575, 2000.}

\bibitem{CD}
A.~Cohen and R.~DeVore.
\newblock {\it Approximation of high-dimensional parametric pdes},
\newblock  Acta Numerica, 24,1-159, 2015.

\bibitem{CDS} A. Cohen, R. DeVore and C. Schwab,
{\it Analytic regularity and polynomial approximation of parametric and stochastic PDEs}, Analysis and Applications, 9, 11-47, 2011.

\bibitem{DPW}
R.~DeVore, G.~Petrova, and P.~Wojtaszczyk.
\newblock {\it Greedy algorithms for reduced bases in Banach spaces},
\newblock Constructive Approximation, 37(3), 455-466, 2013.

\bibitem{GE}\ac{J. Galvis and Y. Efendiev, \it Domain decomposition preconditioners for multiscale flows 
in high contrast media: reduced dimension coarse spaces,  
SIAM Journal on Multiscale Modeling and Simulation 8(4),1461-1483, 2010.}

\bibitem{Ha} \ac{B. Haasdonk, {\it Reduced basis methods for parametrized PDEs - a tutorial introduction for stationary and instationary problems},
in Model Reduction and Approximation - Theory and Algorithms, P. Benner, A. Cohen, M. Ohlberger, and K. Willcox eds, SIAM, 2017.}

\bibitem{JKO}
\ac{
V.V. Jikov, S.M. Kozlov, and O.A. OIeinik, {\it Homogeneization of differential operators and integral functionals},
Springer, 1994.
}

\bibitem{zou2019adaptive}
D.~Kouri, Z.~Zou and W.~Aquino.
\newblock {\it An adaptive local reduced basis method for solving pdes with
  uncertain inputs and evaluating risk},
\newblock Computer Methods in Applied Mechanics and Engineering
  345, 302--322, 2019.

\bibitem{MPPY}
Y. Maday,  A.T. Patera, J.D. Penn and M. Yano, {\em A parametrized-background data-weak approach to variational data assimilation: Formulation, analysis, and application to acoustics}, Int. J. Numer.  Meth. Eng. 102, 933-965, 2015. DOI: 10.1002/nme.4747

\bibitem{maday2013}
Y.~Maday and B.~Stamm.
\newblock {\it Locally adaptive greedy approximations for anisotropic parameter
  reduced basis spaces},
\newblock SIAM Journal on Scientific Computing, 35(6), A2417-A2441, 2013.

\bibitem{eftang2010}
A.T.~Patera J.L.~Eftang and E.M. Ronquist.
\newblock {\it An ``hp'' certified reduced basis method for parametrized elliptic
  partial differential equations},
\newblock SIAM Journal on Scientific Computing, 32(6), 3170-3200, 2010.

\bibitem{P} A. Pinkus,
{\it $N$-widths in approximation theory}, Springer, 1985.

\bibitem{RHP} G. Rozza, D.B.P. Huynh, and A.T. Patera, {\em Reduced basis approximation and a posteriori error estimation for affinely parametrized elliptic coercive partial differential equations - application to transport and continuum mechanics}, Archive of Computational Methods in Engineering 15, 229-275, 2008. DOI: 10.1007/s00791-006-0044-7

\bibitem{S} S. Sen, {\it Reduced-basis approximation and a posteriori error estimation for many-parameter heat conduction problems}, Numerical Heat Transfer B-Fund 54, 369-389, 2008. DOI: 10.1080/10407790802424204

\bibitem{St}\ac{E.M. Stein, {\it Singular integrals and differentiability properties of functions}, Princeton University Press, 1970.}

\bibitem{Tem}
V. Temlyakov
\newblock {\it Nonlinear Kolmogorov widths},
\newblock Math. Notes 63, 785-795, 1998.

\bibitem{TWZ} H. Tran, C.G. Webster, and G. Zhang, {\it Analysis of quasi-optimal polynomial approximations for parameterized PDEs with deterministic and stochastic coefficients}, Numerische Mathematik 137, 451-493, 2017.

\bibitem{VPRP}\ac{ K. Veroy, C. Prudhomme, D.V. Rovas and T. Patera, {\it A Posteriori Error Bounds for Reduced-Basis
Approximation of Parametrized Noncoercive and Nonlinear Elliptic Partial Differential Equations}, Proc. 16th AIAA Computational Fluid Dynamics Conference,
Orlando, 2003.}

\bibitem{Vo} \ac{S. Volkwein, {\it Proper Orthogonal Decomposition: Theory and reduced order modeling.}, Lecture Notes, University of Konstanz, 2013.}

\bibitem{WP}\ac{K. Willcox and J. Peraire, {\it Balanced model reduction via the proper orthogonal decomposition}, American Institute of Aeronautics and Astronautics 40, 2323-2330, 2022. }

\bibitem{Zen} \ac{ A. Zenisek, {\it extensions from the Sobolev spaces $H^1$ satisfying prescribed Dirichlet boundary conditions}, Applications of
Mathematics 49, 405-413, 2004.}

  

\end{thebibliography}
\end{document}